\documentclass[oneside,reqno]{amsart}
\usepackage[T1]{fontenc}
\usepackage[latin9]{inputenc}
\usepackage{prettyref}
\usepackage{float}
\usepackage{amsthm}
\usepackage{amstext}
\usepackage{amssymb}
\usepackage{esint}
\usepackage{xargs}[2008/03/08]
\usepackage[unicode=true,pdfusetitle,
 bookmarks=true,bookmarksnumbered=false,bookmarksopen=false,
 breaklinks=false,pdfborder={0 0 1},backref=false,colorlinks=false]
 {hyperref}
\usepackage{breakurl}

\makeatletter

\floatstyle{ruled}
\newfloat{algorithm}{tbp}{loa}
\providecommand{\algorithmname}{Algorithm}
\floatname{algorithm}{\protect\algorithmname}

\numberwithin{equation}{section}
\numberwithin{figure}{section}
\theoremstyle{plain}
\newtheorem{thm}{\protect\theoremname}[section]
  \theoremstyle{definition}
  \newtheorem{defn}[thm]{\protect\definitionname}
  \theoremstyle{definition}
  \newtheorem{problem}[thm]{\protect\problemname}
  \theoremstyle{definition}
  \newtheorem{example}[thm]{\protect\examplename}
  \theoremstyle{remark}
  \newtheorem{rem}[thm]{\protect\remarkname}
  \theoremstyle{plain}
  \newtheorem{prop}[thm]{\protect\propositionname}
  \theoremstyle{plain}
  \newtheorem{cor}[thm]{\protect\corollaryname}
  \theoremstyle{plain}
  \newtheorem{lem}[thm]{\protect\lemmaname}

\usepackage{comment}
\usepackage{amssymb,yhmath}
\usepackage{etex}

\newcommand{\isdef}{\ensuremath{\stackrel{\text{def}}{=}}}

\newcommand{\RR}{\ensuremath{\mathbb{R}}}
\newcommand{\CC}{\ensuremath{\mathbb{C}}}

\DeclareMathOperator{\dd}{d}

\DeclareMathOperator{\diag}{diag}
\DeclareMathOperator{\spann}{span}

\DeclareMathOperator{\rank}{rank}

\usepackage{prettyref}
\newrefformat{prob}{Problem \ref{#1}}
\newrefformat{prop}{Proposition \ref{#1}}
\newrefformat{def}{Definition \ref{#1}}
\newrefformat{sub}{Section \ref{#1}}
\newrefformat{alg}{Algorithm \ref{#1}}
\newrefformat{con}{Conjecture \ref{#1}}
\newrefformat{cor}{Corollary \ref{#1}}
\newrefformat{rem}{Remark \ref{#1}}
\newrefformat{app}{Appendix \ref{#1}}
\newrefformat{ex}{Example \ref{#1}}
\newrefformat{enu}{Item \ref{#1}}

\DeclareRobustCommand*\cal{\@fontswitch\relax\mathcal}

\makeatother

  \providecommand{\corollaryname}{Corollary}
  \providecommand{\definitionname}{Definition}
  \providecommand{\examplename}{Example}
  \providecommand{\lemmaname}{Lemma}
  \providecommand{\problemname}{Problem}
  \providecommand{\propositionname}{Proposition}
  \providecommand{\remarkname}{Remark}
\providecommand{\theoremname}{Theorem}

\begin{document}

\title{Geometry and Singularities of the Prony mapping}

\author{Dmitry Batenkov}

\address{Department of Mathematics\\
Weizmann Institute of Science\\
Rehovot 76100\\
Israel}

\email{dima.batenkov@weizmann.ac.il}

\urladdr{http://www.wisdom.weizmann.ac.il/~dmitryb}

\author{Yosef Yomdin}

\email{yosef.yomdin@weizmann.ac.il}

\urladdr{http://www.wisdom.weizmann.ac.il/~yomdin}

\keywords{{\small Singularities, Signal acquisition, Non-linear models, Moments
inversion.}}

\subjclass[2000]{{\small 94A12 62J02, 14P10, 42C99}}

\thanks{This research is supported by the Adams Fellowship Program of the
Israel Academy of Sciences and Humanities, ISF grant 264/09 and the
Minerva Foundation.}
\begin{abstract}
Prony mapping provides the global solution of the Prony system of
equations 
\[
\Sigma_{i=1}^{n}A_{i}x_{i}^{k}=m_{k},\ k=0,1,\dots,2n-1.
\]
This system appears in numerous theoretical and applied problems arising
in Signal Reconstruction. The simplest example is the problem of reconstruction
of linear combination of $\delta$-functions of the form $g(x)=\sum_{i=1}^{n}a_{i}\delta(x-x_{i})$,
with the unknown parameters $a_{i},\ x_{i},\ i=1,\dots,n,$ from the
``moment measurements'' $m_{k}=\int x^{k}g(x)dx.$

Global solution of the Prony system, i.e. inversion of the Prony mapping,
encounters several types of singularities. One of the most important
ones is a collision of some of the points $x_{i}.$ The investigation
of this type of singularities has been started in \cite{yom2009Singularities}
where the role of finite differences was demonstrated.

In the present paper we study this and other types of singularities
of the Prony mapping, and describe its global geometry. We show, in
particular, close connections of the Prony mapping with the ``Vieta
mapping'' expressing the coefficients of a polynomial through its
roots, and with hyperbolic polynomials and ``Vandermonde mapping''
studied by V. Arnold.
\end{abstract}
\maketitle
\global\long\def\g{\gamma}
\global\long\def\M{\tilde{M}}
\global\long\def\PM{{\cal PM}}
\global\long\def\err{\varepsilon}
\newcommandx\confvec[2][usedefault, addprefix=\global, 1=j, 2=k]{\vec{v_{#1,#2}}}
 \global\long\def\cvand{\ensuremath{V}}
 \global\long\def\o{\omega}
\global\long\def\SM{{\cal SM}}
\global\long\def\TM{{\cal TM}}
\global\long\def\m{\mu}
\global\long\def\e{\epsilon}

\global\long\def\vec#1{\ensuremath{\mathbf{#1}}}

\section{Introduction}

Prony system appears as we try to solve a very simple ``algebraic
signal reconstruction'' problem of the following form: assume that
the signal $F(x)$ is known to be a linear combination of shifted
$\delta$-functions:
\begin{equation}
F\left(x\right)=\sum_{j=1}^{d}a_{j}\delta\left(x-x_{j}\right).\label{eq:equation_model_delta}
\end{equation}
We shall use as measurements the polynomial moments:
\begin{equation}
m_{k}=m_{k}\left(F\right)=\int x^{k}F\left(x\right)\dd x.\label{eq:moments}
\end{equation}
After substituting $F$ into the integral defining $m_{k}$ we get
\[
m_{k}(F)=\int x^{k}\sum_{j=1}^{d}a_{j}\delta(x-x_{j})\dd x=\sum_{j=1}^{d}a_{j}x_{j}^{k}.
\]
Considering $a_{j}$ and $x_{j}$ as unknowns, we obtain equations
\begin{equation}
m_{k}\left(F\right)=\sum_{j=1}^{d}a_{j}x_{j}^{k},\; k=0,1,\dots.\label{eq:equation_prony_system}
\end{equation}
This infinite set of equations (or its part, for $k=0,1,\dots,2d-1$),
is called Prony system. It can be traced at least to R. de Prony (1795,
\cite{prony1795essai}) and it is used in a wide variety of theoretical
and applied fields. See \cite{Auton1981} for an extensive bibligoraphy
on the Prony method.

In writing Prony system \eqref{eq:equation_prony_system} we have
assumed that all the nodes $x_{1},\dots,x_{d}$ are pairwise different.
However, as a right-hand side $\mu=(m_{0},\dots,m_{2d-1})$ of \eqref{eq:equation_prony_system}
is provided by the actual measurements of the signal $F$, we cannot
guarantee a priori, that this condition is satisfied for the solution.
Moreover, we shall see below that multiple nodes may naturally appear
in the solution process. In order to incorporate possible collisions
of the nodes, we consider ``confluent Prony systems''.

Assume that the signal $F(x)$ is a linear combination of shifted
$\delta$-functions and their derivatives:
\begin{equation}
F\left(x\right)=\sum_{j=1}^{s}\sum_{\ell=0}^{d_{j}-1}a_{j,\ell}\delta^{\left(\ell\right)}\left(x-x_{j}\right).\label{eq:confluent equation_model_delta}
\end{equation}

\begin{defn}
\label{def:mult-vec-signal}For $F\left(x\right)$ as above, the vector
$D\left(F\right)\isdef(d_{1},\dots,d_{s})$ is \emph{the multiplicity
vector} of $F$, $s=s\left(F\right)$ is the size of its support,
$T\left(F\right)\isdef\left(x_{1},\dots,x_{s}\right)$, and $\rank\left(F\right)\isdef\sum_{j=1}^{s}d_{j}$
is its rank. For avoiding ambiguity in these definitions, it is always
understood that $a_{j,d_{j}-1}\neq0$ for all $j=1,\dots,s$ (i.e.
$d_{j}$ is the maximal index for which $a_{j,d_{j}-1}\neq0$).
\end{defn}
For the moments $m_{k}=m_{k}(F)=\int x^{k}F(x)\dd x$ we now get 
\[
m_{k}=\sum_{j=1}^{s}\sum_{\ell=0}^{d_{j}-1}a_{j,\ell}\frac{{k!}}{{(k-\ell)!}}x_{j}^{k-\ell}.
\]
Considering $x_{i}$ and $a_{j,\ell}$ as unknowns, we obtain a system
of equations
\begin{equation}
\sum_{j=1}^{s}\sum_{\ell=0}^{d_{j}-1}\frac{k!}{\left(k-\ell\right)!}a_{j,\ell}x_{j}^{k-\ell}=m_{k},\quad k=0,1,\dots,2d-1,\label{eq:confluent equation_prony_system}
\end{equation}
which is called a confluent Prony system of order $d$ with the multiplicity
vector $D=\left(d_{1},\dots,d_{s}\right)$. The original Prony system
\eqref{eq:equation_prony_system} is a special case of the confluent
one, with $D$ being the vector $(1,\dots,1)$ of length $d$.

The system \eqref{eq:confluent equation_prony_system} arises also
in the problem of reconstructing a planar polygon $P$ (or even an
arbitrary semi-analytic \emph{quadrature domain}) from its moments
\[
m_{k}(\chi_{P})=\iint_{\RR^{2}}z^{k}\chi_{P}\dd x\dd y,\; z=x+\imath y,
\]
where $\chi_{P}$ is the characteristic function of the domain $P\subset\RR^{2}$.
This problem is important in many areas of science and engineering
\cite{gustafsson2000rpd}. The above yields the confluent Prony system
\[
m_{k}=\sum_{j=1}^{s}\sum_{i=0}^{d_{j}-1}c_{i,j}k(k-1)\cdots(k-i+1)z_{j}^{k-i},\qquad c_{i,j}\in\CC,\; z_{j}\in\CC\setminus\left\{ 0\right\} .
\]

\begin{defn}
For a given multiplicity vector $D=\left(d_{1},\dots,d_{s}\right)$,
its \emph{order} is $\sum_{j=1}^{s}d_{j}$.
\end{defn}
As we shall see below, if we start with the measurements $\mu(F)=\mu=(m_{0},\dots,m_{2d-1})$,
then a natural setting of the problem of solving the Prony system
is the following:
\begin{problem}[Prony problem of order $d$]
\label{prob:prony}\textit{Given the measurements 
\[
\mu=(m_{0},\dots,m_{2d-1})\in\CC^{2d}
\]
in the right hand side of \eqref{eq:confluent equation_prony_system},
find the multiplicity vector $D=(d_{1},\dots,d_{s})$ of order $r=\sum_{j=1}^{s}d_{j}\leq d$,
and find the unknowns $x_{j}$ and $a_{j,\ell},$ which solve the
corresponding confluent Prony system \eqref{eq:confluent equation_prony_system}
with the multiplicity vector $D$ (hence, with solution of rank $r$).}
\end{problem}
It is extremely important in practice to have \emph{a stable method
of inversion}. Many research efforts are devoted to this task (see
e.g. \cite{badeau2008performance,batenkov2011accuracy,donoho2006stable,peter2011nonlinear,potts2010parameter,stoica1989music}
and references therein). A basic question here is the following.
\begin{problem}[Noisy Prony problem]
\label{prob:noisy-prony}Given the \emph{noisy} measurements \textit{
\[
\tilde{\mu}=(\tilde{m_{0}},\dots,\tilde{m}_{2d-1})\in\CC^{2d}
\]
and an estimate of the error $\left|\tilde{m}_{k}-m_{k}\right|\leq\err_{k}$,
solve \prettyref{prob:prony} so as to minimize the reconstruction
error.}
\end{problem}
In this paper we study the global setting of the Prony problem, stressing
its algebraic structure. In \prettyref{sec:mappings-def} the space
where the solution is to be found (Prony space) is described. It turns
out to be a vector bundle over the space of the nodes $x_{1},\dots,x_{d}$.
We define also three mappings: ``Prony'', ``Taylor'', and ``Stieltjes''
ones, which capture the essential features of the Prony problem and
of its solution process.

In \prettyref{sec:solving-prony} we investigate solvability conditions
for the Prony problem. The answer leads naturally to a stratification
of the space of the right-hand sides, according to the rank of the
associated Hankel-type matrix and its minors. The behavior of the
solutions near various strata turns out to be highly nontrivial, and
we present some initial results in the description of the corresponding
singularities.

In \prettyref{sec:multiplicity-restricted}, we study the multiplicity-restricted
Prony problem, fixing the collision pattern of the solution, and derive
simple bounds for the stability of the solution via factorization
of the Jacobian determinant of the corresponding Prony map.

In \prettyref{sec:rank-restriction} we consider the rank-restricted
Prony problem, effectively reducing the dimension to $2r$ instead
of $2d$, where $r$ is precisely the rank of the associated Hankel-type
matrix. In this formulation, the Prony problem is solvable in a small
neighborhood of the exact measurement vector.

In \prettyref{sec:collision-singularities} we study one of the most
important singularities in the Prony problem: collision of some of
the points $x_{i}.$ The investigation of this type of singularities
has been started in \cite{yom2009Singularities} where the role of
finite differences was demonstrated. In the present paper we introduce
global bases of finite differences, study their properties, and prove
that using such bases we can resolve in a robust way at least the
linear part of the Prony problem at and near colliding configurations
of the nodes.

In \prettyref{sec:real-prony} we discuss close connections of the
Prony problem with hyperbolic polynomials and ``Vandermonde mapping''
studied by V.I.Arnold in \cite{arnold1986hpa} and by V.P.Kostov in
\cite{kostov1989geometric,kostov2006root,kostov2007root}, and with
``Vieta mapping'' expressing the coefficients of a polynomial through
its roots. We believe that questions arising in theoretical study
of Prony problem and in its practical applications justify further
investigation of these connections, as well as further applications
of Singularity Theory.

Finally, in \prettyref{app:proof-pade-solvability} we describe a
solution method for the Prony system based on Padé approximation.

\section{\label{sec:mappings-def}Prony, Stieltjes and Taylor Mappings}

In this section we define ``Prony'', ``Taylor'', and ``Stieltjes''
mappings, which capture some essential features of the Prony problem
and of its solution process. The main idea behind the spaces and mappings
introduced in this section is the following: associate to the signal
$F(x)=\sum_{i=1}^{d}a_{i}\delta(x-x_{i})$ the rational function $R(z)=\sum_{i=1}^{d}\frac{{a_{i}}}{{z-x_{i}}}$.
(In fact, $R$ is the Stieltjes integral transform of $F$). The functions
$R$ obtained in this way can be written as $R(z)=\frac{{P(z)}}{{Q(z)}}$
with $\deg P\leq\deg Q-1,$ and they satisfy $R(\infty)=0$. Write
$R$ as $R(z)=\sum_{i=1}^{d}\frac{{za_{i}}}{{1-x_{i}/z}}.$ Developing
the summands into geometric progressions we conclude that $R(z)=\sum_{k=0}^{\infty}m_{k}(\frac{1}{z})^{k+1},$
with $m_{k}=\sum_{i=1}^{d}a_{i}x_{i}^{k}$, so the moment measurements
$m_{k}$ in the right hand side of the Prony system (\ref{eq:equation_prony_system})
are exactly the Taylor coefficients of $R(z)$. We shall see below
that this correspondence reduces solution of the Prony system to an
appropriate Padé approximation problem.
\begin{defn}
\label{def:config}For each $w=\left(x_{1},\dots,x_{d}\right)\in\CC^{d}$,
let $s=s\left(w\right)$ be the number of distinct coordinates $\tau_{j}$,
$j=1,\dots,s$, and denote $T\left(w\right)=\left(\tau_{1},\dots,\tau_{s}\right)$.
The multiplicity vector is $D=D\left(w\right)=\left(d_{1},\dots,d_{s}\right)$,
where $d_{j}$ is the number of times the value $\tau_{j}$ appears
in $\left\{ x_{1},\dots,x_{d}\right\} .$ The order of the values
in $T\left(w\right)$ is defined by their order of appearance in $w$.\end{defn}
\begin{example}
For $w=\left(3,1,2,1,0,3,2\right)$ we have $s\left(w\right)=4$,
$T\left(w\right)=\left(3,1,2,0\right)$ and $D\left(w\right)=\left(2,2,2,1\right)$.\end{example}
\begin{rem}
\label{rem:mult-vec-abuse}Note the slight abuse of notations between
\prettyref{def:mult-vec-signal} and \prettyref{def:config}. Note
also that the \emph{order }of \emph{$D\left(w\right)$ }equals to\emph{
$d$ }\textbf{\emph{for all }}$w\in\CC^{d}$\emph{.}\end{rem}
\begin{defn}
\label{def:bundles}For each $w\in\CC^{d}$, let $s=s\left(w\right),\; T\left(w\right)=\left(\tau_{1},\dots,\tau_{s}\right)$
and $D\left(w\right)=\left(d_{1},\dots,d_{s}\right)$ be as in \prettyref{def:config}.
\begin{enumerate}
\item $V_{w}$ is the vector space of dimension $d$ containing the linear
combinations
\begin{equation}
g=\sum_{j=1}^{s}\sum_{\ell=0}^{d_{j}-1}\g_{j,\ell}\delta^{\left(\ell\right)}\left(x-\tau_{j}\right)\label{eq:standard-basis-representation}
\end{equation}
of $\delta$-functions and their derivatives at the points of $T\left(w\right)$.
The ``standard basis'' of $V_{w}$ is given by the distributions
\begin{equation}
\delta_{j,\ell}=\delta^{\left(\ell\right)}\left(x-\tau_{j}\right),\qquad j=1,\dots,s\left(w\right);\;\ell=0,\dots,d_{j}-1.\label{eq:standard-basis-prony}
\end{equation}

\item $W_{w}$ is the vector space of dimension $d$ of all the rational
functions with poles $T\left(w\right)$ and multiplicities $D\left(w\right)$,
vanishing at $\infty:$
\[
R\left(z\right)=\frac{P\left(z\right)}{Q\left(z\right)},\qquad Q\left(z\right)=\prod_{j=1}^{s}\left(z-\tau_{j}\right)^{d_{j}},\;\deg P\left(z\right)<\deg Q\leqslant d.
\]
The ``standard basis'' of $W_{w}$ is given by the elementary fractions
\[
R_{j,\ell}=\frac{1}{\left(z-\tau_{j}\right)^{\ell}},\qquad j=1,\dots,s;\;\ell=1,\dots,d_{j}.
\]

\end{enumerate}
\end{defn}
\begin{minipage}[t]{1\columnwidth}%
\end{minipage}

Now we are ready to formally define the Prony space ${\cal P}_{d}$
and the Stieltjes space ${\cal S}_{d}$.
\begin{defn}
The Prony space ${\cal P}_{d}$ is the vector bundle over $\CC^{d}$,
consisting of all the pairs
\[
\left(w,g\right):\quad w\in\CC^{d},\; g\in V_{w}.
\]
The topology on ${\cal P}_{d}$ is induced by the natural embedding
${\cal P}_{d}\subset\CC^{d}\times{\cal D},$ where ${\cal D}$ is
the space of distributions on $\CC$ with its standard topology.
\end{defn}
\begin{minipage}[t]{1\columnwidth}%
\end{minipage}
\begin{defn}
The Stieltjes space ${\cal S}_{d}$ is the vector bundle over $\CC^{d}$,
consisting of all the pairs
\[
\left(w,\g\right):\qquad w\in\CC^{d},\;\g\in W_{w}.
\]
The topology on ${\cal S}_{d}$ is induced by the natural embedding
${\cal S}_{d}\subset\CC^{d}\times{\cal R}$, where ${\cal R}$ is
the space of complex rational functions with its standard topology.
\end{defn}
\begin{minipage}[t]{1\columnwidth}%
\end{minipage}
\begin{defn}
\label{def:stieltjes-mapping}The Stieltjes mapping $\SM:{\cal P}_{d}\rightarrow{\cal S}_{d}$
is defined by the Stieltjes integral transform: for $(w,g)\in{\cal P}_{d}$
\[
\SM\left(\left(w,g\right)\right)=\left(w,\g\right),\qquad\g\left(z\right)=\int_{-\infty}^{\infty}\frac{g\left(x\right)\dd x}{z-x}.
\]
Sometimes we abuse notation and write for short $\SM\left(g\right)=\g$,
with the understanding that $\SM$ is also a map $\SM:V_{w}\to W_{w}$
for each $w\in\CC^{d}$.
\end{defn}
The following fact is immediate consequence of the above definitions.
\begin{prop}
$\SM$ is a linear isomorphism of the bundles ${\cal P}_{d}$ and
${\cal S}_{d}$ (for each $w\in\CC^{d}$, $\SM$ is a linear isomorphism
of the vector spaces $V_{w}$ and $W_{w}$). In the standard bases
of $V_{w}$ and $W_{w}$, the map $\SM$ is diagonal, satisfying
\[
\SM\left(\delta_{j,\ell}\right)=\left(-1\right)^{\ell}\ell!R_{j,\ell}\left(z\right).
\]
Furthermore, for any $\left(w,g\right)\in{\cal P}_{d}$ 
\begin{equation}
\SM\left(g\right)=\underbrace{\frac{P\left(z\right)}{Q\left(z\right)}}_{\text{irreducible}},\qquad\deg P<\deg Q=\rank\left(g\right)\leqslant d.\label{eq:sm-rank-equals-degree-denominator}
\end{equation}
\end{prop}
\begin{defn}
The Taylor space ${\cal T}_{d}$ is the space of complex Taylor polynomials
at infinity of degree $2d-1$ of the form $\sum_{k=0}^{2d-1}m_{k}(\frac{1}{z})^{k+1}$.
We shall identify ${\cal T}_{d}$ with the complex space $\CC^{2d}$
with the coordinates $m_{0},\dots,m_{2d-1}$.
\end{defn}
\begin{minipage}[t]{1\columnwidth}%
\end{minipage}
\begin{defn}
The Taylor mapping $\TM:{\cal S}_{d}\to{\cal T}_{d}$ is defined by
the truncated Taylor development at infinity:
\[
\TM\left(\left(w,\g\right)\right)=\sum_{k=0}^{2d-1}\alpha_{k}\left(\frac{1}{z}\right)^{k+1},\qquad\text{ where }\g\left(z\right)=\sum_{k=0}^{\infty}\alpha_{k}\left(\frac{1}{z}\right)^{k+1}.
\]
We identify $\TM\left(\left(w,\g\right)\right)$ as above with $\left(\alpha_{0},\dots,\alpha_{2d-1}\right)\in\CC^{2d}.$
Sometimes we write for short $\TM\left(\g\right)=\left(\alpha_{0},\dots,\alpha_{2d-1}\right)$.
\end{defn}
Finally, we define the Prony mapping ${\cal PM}$ which encodes the
Prony problem.
\begin{defn}
The Prony mapping $\PM:{\cal P}_{d}\to\CC^{2d}$ for $\left(w,g\right)\in{\cal P}_{d}$
is defined as follows:
\[
{\cal \PM}\left(\left(w,g\right)\right)=\left(m_{0},\dots,m_{2d-1}\right)\in\CC^{2d},\qquad m_{k}=m_{k}\left(g\right)=\int x^{k}g\left(x\right)\dd x.
\]

\end{defn}
By the above definitions, we have
\begin{equation}
\PM=\TM\circ\SM.\label{eq:prony-mapping-factor}
\end{equation}

\noindent Solving the Prony problem for a given right-hand side $(m_{0},\dots,m_{2d-1})$
is therefore equivalent to inverting the Prony mapping $\PM$. As
we shall elaborate in the subsequent section, the identity \eqref{eq:prony-mapping-factor}
allows us to split this problem into two parts: inversion of $\TM$,
which is, essentially, the Padé approximation problem, and inversion
of $\SM$, which is, essentially, the decomposition of a given rational
function into the sum of elementary fractions.

\section{\label{sec:solving-prony}Solvability of the Prony problem}

\subsection{General condition for solvability}

In this section we provde a necessary and sufficient condition for
the Prony problem to have a solution (which is unique, as it turns
out by \prettyref{prop:pade-solution-unique}). As mentioned in the
end of the previous section, our method is based on inverting \eqref{eq:prony-mapping-factor}
and thus relies on the solution of the corresponding (diagonal) \emph{Padé
approximation problem} \cite{baker1981pap}. 
\begin{problem}[Diagonal Padé approximation problem]
\label{prob:pade}\textit{Given $\mu=\left(m_{0},\dots,m_{2d-1}\right)\in\CC^{2d}$,
find a rational function $R_{d}(z)=\frac{P\left(z\right)}{Q\left(z\right)}\in{\cal S}_{d}$
with $\deg P<\deg Q\leqslant d$, such that the first $2d$ Taylor
coefficients at infinity of $R_{d}(z)$ are $\left\{ m_{k}\right\} _{k=0}^{2d-1}$.}\end{problem}
\begin{prop}
\label{prop:pade-solution-unique}A solution to \prettyref{prob:pade},
if exists, is unique.\end{prop}
\begin{proof}
Writing $R\left(z\right)=\frac{P\left(z\right)}{Q\left(z\right)},\; R_{1}\left(z\right)=\frac{P_{1}\left(z\right)}{Q_{1}\left(z\right)}$,
with $\deg P<\deg Q\leqslant d$ and $\deg P_{1}<\deg Q_{1}\leqslant d$,
we get
\[
R-R_{1}=\frac{PQ_{1}-P_{1}Q}{QQ_{1}},
\]
and this function, if nonzero, can have a zero of order at most $2d-1$
at infinity.
\end{proof}
Let us summarize the above discussion with the following statement.
\begin{prop}
\label{prop:pade-correspondence}The tuple 
\[
\left\{ s,\; D=(d_{1},\dots,d_{s}),\; r=\sum_{j=1}^{s}d_{j}\leq d,\; X=\left\{ x_{j}\right\} _{j=1}^{s},\; A=\left\{ a_{j,\ell}\right\} _{j=1,\dots,s;\;\ell=0,\dots,d_{j}-1}\right\} 
\]
is a (unique, up to a permutation of the nodes $\left\{ x_{j}\right\} $)
solution to \prettyref{prob:prony} with right-hand side 
\[
\mu=(m_{0},\dots,m_{2d-1})\in\CC^{2d}
\]
if and only if the rational function 
\[
R_{D,X,A}\left(z\right)=\sum_{j=1}^{s}\sum_{\ell=1}^{d_{j}}\left(-1\right)^{\ell-1}\left(\ell-1\right)!\frac{a_{j,\ell-1}}{\left(z-x_{j}\right)^{\ell}}=\sum_{k=0}^{2d-1}\frac{m_{k}}{z^{k+1}}+O\left(z^{-2d-1}\right)
\]
is a (unique) solution to \prettyref{prob:pade} with input $\mu$.
In that case,
\[
R_{D,X,A}\left(z\right)=\int_{-\infty}^{\infty}\frac{g\left(x\right)\dd x}{z-x}\qquad\text{ where }\; g\left(x\right)=\sum_{j=1}^{s}\sum_{\ell=0}^{d_{j}-1}a_{j,\ell}\delta^{\left(\ell\right)}\left(x-x_{j}\right),
\]
i.e. $R_{D,X,A}\left(z\right)$ is the Stieltjes transform of $g\left(x\right)$.\end{prop}
\begin{proof}
This follows from the definitions of \prettyref{sec:mappings-def},
\eqref{eq:prony-mapping-factor}, \prettyref{prop:pade-solution-unique}
and the fact that the problem of representing a given rational function
as a sum of elementary fractions of the specified form (i.e. inverting
$\SM$) is always uniquely solvable up to a permutation of the poles.
\end{proof}
The next result provides necessary and sufficient conditions for the
solvability of \prettyref{prob:pade}. It summarizes some well-known
facts in the theory of Padé approximation, related to ``normal indices''
(see, for instance, \cite{baker1981pap}). However, these facts are
not usually formulated in the literature on Padé approximation in
the form we need in relation to the Prony problem. Consequently, we
give a detailed proof of this result in \prettyref{app:proof-pade-solvability}.
This proof contains, in particular, some facts which are important
for understanding the solvability issues of the Prony problem.
\begin{defn}
\label{def:hankel-and-minors}Given a vector $\mu=\left(m_{0},\dots,m_{2d-1}\right)$,
let $\M_{d}$ denote the $d\times\left(d+1\right)$ Hankel matrix
\begin{equation}
\M_{d}=\begin{bmatrix}m_{0} & m_{1} & m_{2} & \dots & m_{d}\\
m_{1} & m_{2} & m_{3} & \dots & m_{d+1}\\
\adots & \adots & \adots & \adots & \adots\\
m_{d-1} & m_{d} & m_{d+1} & \dots & m_{2d-1}
\end{bmatrix}.\label{eq:hankel-definition}
\end{equation}
For each $e\leqslant d$, denote by $\M_{e}$ the $e\times\left(e+1\right)$
submatrix of $\M_{d}$ formed by the first $e$ rows and $e+1$ columns,
and let $M_{e}$ denote the corresponding square matrix.\end{defn}
\begin{thm}
\label{thm:pade-solvability}Let $\mu=(m_{0},\dots,m_{2d-1})$ be
given, and let $r\leqslant d$ be the rank of the Hankel matrix $\M_{d}$
as in \eqref{eq:hankel-definition}. Then \prettyref{prob:pade} is
solvable for the input $\mu$ if and only if the upper left minor
$\left|M_{r}\right|$ of $\M_{d}$ is non-zero.
\end{thm}
As an immediate consequence of \prettyref{thm:pade-solvability} and
\prettyref{prop:pade-correspondence}, we obtain the following result.
\begin{thm}
\label{thm:prony-solvability}Let $\mu=(m_{0},\dots,m_{2d-1})$ be
given, and let $r\leqslant d$ be the rank of the Hankel matrix $\M_{d}$
as in \eqref{eq:hankel-definition}. Then \prettyref{prob:prony}
with input $\mu$ is solvable if and only if the upper left minor
$\left|M_{r}\right|$ of $\M_{d}$ is non-zero. The solution, if exists,
is unique, up to a permutation of the nodes $\left\{ x_{j}\right\} $.
The multiplicity vector $D=\left(d_{1},\dots,d_{s}\right)$, of order
$\sum_{j=1}^{s}d_{j}=r$, of the resulting confluent Prony system
of rank $r$ is the multiplicity vector of the poles of the rational
function $R_{D,X,A}\left(z\right)$, solving the corresponding Padé
problem.
\end{thm}
As a corollary we get a complete description of the right-hand side
data $\mu\in\CC^{2d}$ for which the Prony problem is solvable (unsolvable).
Define for $r=1,\dots,d$ sets $\Sigma_{r}\subset\CC^{2d}$ (respectively,
$\Sigma'_{r}\subset\CC^{2d}$) consisting of $\mu\in\CC^{2d}$ for
which the rank of $\tilde{M}_{d}=r$ and $|M_{r}|\ne0$ \ (respectively,
$|M_{r}|=0)$. The set $\Sigma_{r}$ is a difference $\Sigma_{r}=\Sigma_{r}^{1}\setminus\Sigma_{r}^{2}$
of two algebraic sets: $\Sigma_{r}^{1}$ is defined by vanishing of
all the $s\times s$ minors of $\tilde{M}_{d},\ r<s\leq d,$ while
$\Sigma_{r}^{2}$ is defined by vanishing of $|M_{r}|.$ In turn,
$\Sigma'_{r}=\Sigma_{r}^{'1}\setminus\Sigma_{r}^{'2},$ with $\Sigma_{r}^{'1}=\Sigma_{r}^{1}\cap\Sigma_{r}^{2}$
and $\Sigma_{r}^{'2}$ defined by vanishing of all the $r\times r$
minors of $\tilde{M}_{d}.$ The union $\Sigma_{r}\cup\Sigma'_{r}$
consists of all $\mu$ for which the rank of $\tilde{M}_{d}=r,$ which
is $\Sigma_{r}^{1}\setminus\Sigma_{r}^{'2}.$
\begin{cor}
\label{cor:non.solv.set}The set $\Sigma$ (respectively, $\Sigma'$)
of $\mu\in\CC^{2d}$ for which the Prony problem is solvable (respectively,
unsolvable) is the union $\Sigma=\cup_{r=1}^{d}\Sigma_{r}$ (respectively,
$\Sigma'=\cup_{r=1}^{d}\Sigma'_{r}$). In particular, $\Sigma'\subset\{\mu\in\CC^{2d},\det M_{d}=0\}.$
\end{cor}
So for a generic right hand side $\mu$ we have $|M_{d}|\ne0$, and
the Prony problem is solvable. On the algebraic hypersurface of $\mu$
for which $|M_{d}|=0,$ the Prony problem is solvable if $M_{d-1}\ne0$,
etc.

Let us now consider some examples.
\begin{example}
Let us fix $d=1,2,\dots$. Consider $\mu=(m_{0},\dots,m_{2d-1})\in\CC^{2d}$,
the right hand sides of the Prony problem, to be of the form $\mu=\mu_{\ell}=\left(\delta_{k\ell}\right)=(0,\dots,0,\underbrace{1}_{\text{ position \ensuremath{\ell}+1}},0,\dots,0)$,
with all the $m_{k}=0$ besides $m_{\ell}=1,\ \ell=0,\dots,2d-1,$
and let $\tilde{M}_{d}^{\ell}$ be the corresponding matrix.
\begin{prop}
\label{prop:example-1}The rank of $\tilde{M}_{d}^{\ell}$ is equal
to $\ell+1$ for $\ell\leq d-1$, and it is equal to $2d-\ell$ for
$\ell\geq d$. The corresponding Prony problem is solvable for $\ell\leq d-1$,
and it is unsolvable for $\ell\geq d$.\end{prop}
\begin{proof}
For $d=5$ and $\ell=2,4,5,9$, the corresponding matrices $\M_{\ell}^{d}$
are as follows.

\begin{eqnarray*}
\M_{5}^{2} & = & \begin{bmatrix}0 & 0 & 1 & 0 & 0 & 0\\
0 & 1 & 0 & 0 & 0 & 0\\
1 & 0 & 0 & 0 & 0 & 0\\
0 & 0 & 0 & 0 & 0 & 0\\
0 & 0 & 0 & 0 & 0 & 0
\end{bmatrix},\;\M_{5}^{4}=\begin{bmatrix}0 & 0 & 0 & 0 & 1 & 0\\
0 & 0 & 0 & 1 & 0 & 0\\
0 & 0 & 1 & 0 & 0 & 0\\
0 & 1 & 0 & 0 & 0 & 0\\
1 & 0 & 0 & 0 & 0 & 0
\end{bmatrix},\qquad\text{(solvable)}\\
\M_{5}^{5} & = & \begin{bmatrix}0 & 0 & 0 & 0 & 0 & 1\\
0 & 0 & 0 & 0 & 1 & 0\\
0 & 0 & 0 & 1 & 0 & 0\\
0 & 0 & 1 & 0 & 0 & 0\\
0 & 1 & 0 & 0 & 0 & 0
\end{bmatrix},\;\M_{5}^{9}=\begin{bmatrix}0 & 0 & 0 & 0 & 0 & 0\\
0 & 0 & 0 & 0 & 0 & 0\\
0 & 0 & 0 & 0 & 0 & 0\\
0 & 0 & 0 & 0 & 0 & 0\\
0 & 0 & 0 & 0 & 0 & 1
\end{bmatrix}.\qquad\text{(unsolvable)}
\end{eqnarray*}
In general, the matrices $\M_{d}^{\ell}$ have the same pattern as
in the special cases above, so their rank is $\ell+1$ for $\ell\leqslant d-1$,
and $2d-\ell$ for $\ell\geqslant d$, as stated above. Application
of \prettyref{thm:prony-solvability} completes the proof.
\end{proof}

In fact, $\mu_{\ell}$ is a moment sequence of
\[
F\left(x\right)=\frac{1}{\ell!}\delta^{\left(\ell\right)}\left(x\right),
\]
and this signal belongs to ${\cal P}_{d}$ if and only if $\ell\leqslant d-1$.
In notations of \prettyref{cor:non.solv.set} we have

\begin{eqnarray*}
\begin{aligned}\mu_{\ell} & \in & \Sigma_{\ell+1}, &  & \ell\leqslant d-1,\\
\mu_{\ell} & \in & \Sigma'_{2d-\ell}, &  & \ell\geqslant d.
\end{aligned}
\end{eqnarray*}

\end{example}
It is easy to provide various modifications of the above example.
In particular, for $\mu=\tilde{\mu}_{\ell}=\left(0,\dots,0,1,1,\dots,1\right)$,
the result of \prettyref{prop:example-1} remains verbally true.
\begin{example}
\label{ex:nonsolv-ex2}Another example is provided by $\mu_{\ell_{1},\ell_{2}}$,
with all the $m_{k}=0$ besides $m_{\ell_{1}}=1,\; m_{\ell_{2}}=1,\ 0\leq\ell_{1}<d\leq\ell_{2}\leq2d-1.$
For $\ell_{1}<\ell_{2}-d+1$ the rank of the correspondent matrix
$\tilde{M}_{d}$ is $r=2d+\ell_{1}-\ell_{2}+1$ while $|M_{r}|=0$,
so the Prony problem for such $\mu_{\ell_{1},\ell_{2}}$ is unsolvable.
 For $d=5$ and $\ell_{1}=2,\;\ell_{2}=8$ the matrix is as follows:
\[
\M_{5}^{\left(2,8\right)}=\begin{bmatrix}0 & 0 & 1 & 0 & 0 & 0\\
0 & 1 & 0 & 0 & 0 & 0\\
1 & 0 & 0 & 0 & 0 & 0\\
0 & 0 & 0 & 0 & 0 & 1\\
0 & 0 & 0 & 0 & 1 & 0
\end{bmatrix}.
\]

\end{example}

\subsection{\label{sub:singularities}Near-singular inversion}

The behavior of the inversion of the Prony mapping near the unsolvability
stratum $\Sigma'$ and near the strata where the rank of $\tilde{M}_{d}$
drops, turns out to be pretty complicated. In particular, in the first
case at least one of the nodes tends to infinity. In the second case,
depending on the way the right-hand side $\mu$ approaches the lower
rank strata, the nodes may remain bounded, or some of them may tend
to infinity. In this section we provide one initial result in this
direction, as well as some examples. A comprehensive description of
the inversion of the Prony mapping near $\Sigma'$ and near the lower
rank strata is important both in theoretical study and in applications
of Prony-like systems, and we plan to provide further results in this
direction separately.
\begin{thm}
\label{thm:near.unsolv}As the right-hand side $\mu\in\CC^{2d}\setminus\Sigma'$
approaches a finite point $\mu_{0}\in\Sigma',$ at least one of the
nodes $x_{1},\dots,x_{d}$ in the solution tends to infinity.\end{thm}
\begin{proof}
By assumptions, the components $m_{0},\dots,m_{2d-1}$ of the right-hand
side $\mu=(m_{0},\dots,m_{2d-1})\in\CC^{2d}$ remain bounded as $\mu\rightarrow\mu_{0}$.
By \prettyref{thm:f.d.prony.1}, the finite differences coordinates
of the solution $\PM^{-1}(\mu)$ remain bounded as well. Now, if all
the nodes are also bounded, by compactness we conclude that $\PM^{-1}(\mu)\rightarrow\o\in{\cal P}_{d}.$
By continuity in the distribution space (\prettyref{lem:dd-as-our-guys})
we have $\PM(\o)=\mu_{0}$. Hence the Prony problem with the right-hand
side $\mu_{0}$ has a solution $\o\in{\cal P}_{d},$ in contradiction
with the assumption that $\mu_{0}\in\Sigma'$.\end{proof}
\begin{example}
Let us consider an example: $d=2$ and $\mu_{0}=(0,0,1,0)$. Here
the rank $\ell$ of $\tilde{M}_{2}$ is $2$, and $|M_{2}|=0$, so
by \prettyref{thm:prony-solvability} we have $\mu_{0}\in\Sigma'_{2}\subset\Sigma'$.
Consider now a perturbation $\mu(\e)=(0,\e,1,0)$ of $\mu_{0}$. For
$\e\ne0$ we have $\mu(\e)\in\Sigma_{2}\subset\Sigma,$ and the Prony
system is solvable for $\mu_{\e}$. Let us write an explicit solution:
the coefficients $c_{0},c_{1}$ of the polynomial $Q(z)=c_{0}+c_{1}z+z^{2}$
we find from the system \eqref{eq:double-star}:
\[
\begin{bmatrix}0 & \e\\
\e & 1
\end{bmatrix}\begin{bmatrix}c_{0}\\
c_{1}
\end{bmatrix}=\begin{bmatrix}-1\\
0
\end{bmatrix},
\]
whose solution is $c_{1}=-\frac{1}{\e},\ c_{0}=\frac{1}{{\e^{2}}}.$
Hence the denominator $Q(z)$ of $R(z)$ is $Q(z)=\frac{1}{{\e^{2}}}-\frac{1}{\e}z+z^{2}$,
and its roots are $x_{1}=\frac{{1+\imath\sqrt{3}}}{{2\e}},\ x_{2}=\frac{{1-\imath\sqrt{3}}}{{2\e}}$.
The coefficients $b_{0},b_{1}$ of the numerator $P(z)=b_{0}+b_{1}z$
we find from \eqref{eq:star}:
\[
\begin{bmatrix}0 & 0\\
0 & \e
\end{bmatrix}\begin{bmatrix}-\frac{1}{\e}\\
1
\end{bmatrix}=\begin{bmatrix}b_{1}\\
b_{0}
\end{bmatrix},
\]
i.e. $b_{1}=0,\ b_{0}=\e$. Thus the solution of the associated Padé
problem is
\[
R(z)=\frac{{P(z)}}{{Q(z)}}=\frac{\e}{{(z-x_{1})(z-x_{2})}}=\frac{\e^{2}}{\imath\sqrt{3}}\frac{1}{(z-x_{1})}-\frac{\e^{2}}{\imath\sqrt{3}}\frac{1}{(z-x_{2})}.
\]
 Finally, the (unique up to a permutation) solution of the Prony problem
for $\mu_{\e}$ is 
\[
a_{1}=\frac{\e^{2}}{\imath\sqrt{3}},\ a_{2}=-\frac{\e^{2}}{\imath\sqrt{3}},\quad x_{1}=\frac{{1+\imath\sqrt{3}}}{{2\e}},\ x_{2}=\frac{{1-\imath\sqrt{3}}}{{2\e}}.
\]
As $\e$ tends to zero, the nodes $x_{1},x_{2}$ tend to infinity
while the coefficients $a_{1},a_{2}$ tend to zero.
\end{example}
As it was shown above, for a given $\mu\in\Sigma$ (say, with pairwise
different nodes) the rank of the matrix $\tilde{M}_{d}$ is equal
to the number of the nodes in the solution for which the corresponding
$\delta$-function enters with a non-zero coefficients. So $\mu$
approaches a certain $\mu_{0}$ belonging to a stratum of a lower
rank of $\tilde{M}_{d}$ if and only if some of the coefficients $a_{j}$
in the solution tend to zero. We do not analyze all the possible scenarios
of such a degeneration, noticing just that if $\mu_{0}\in\Sigma',$
i.e., the Prony problem is unsolvable for $\mu_{0}$, then \prettyref{thm:near.unsolv}
remains true, with essentially the same proof. So at least one of
the nodes, say, $x_{j},$ escapes to infinity. Moreover, one can show
that $a_{j}x_{j}^{2d-1}$ cannot tend to zero - otherwise the remaining
linear combination of $\delta$-functions would provide a solution
for $\mu_{0}$.

If $\mu_{0}\in\Sigma,$ i.e., the Prony problem is solvable for $\mu_{0},$
all the nodes may remain bounded, or some $x_{j}$ may escape to infinity,
but in such a way that $a_{j}x_{j}^{2d-1}$ tends to zero.

\section{\label{sec:multiplicity-restricted}Multiplicity-restricted Prony
problem}

\global\long\def\rpm{{\cal \PM}_{D_{0}}^{*}}

Consider \prettyref{prob:noisy-prony} at some point $\mu_{0}\in\Sigma$.
By definition, $\mu_{0}\in\Sigma_{r_{0}}$ for some $r_{0}\leq d$.
Let $\mu_{0}=\PM\left(\left(w_{0},g_{0}\right)\right)$ for some $\left(w_{0},g_{0}\right)\in{\cal P}_{d}$.
Assume for a moment that the multiplicity vector $D_{0}=D\left(g_{0}\right)=\left(d_{1},\dots d_{s_{0}}\right)$,
$\sum_{j=1}^{s_{0}}d_{j}=r_{0}$, has a non-trivial collision pattern,
i.e. $d_{j}>1$ for at least one $j=1,\dots,s_{0}$. It means, in
turn, that the function $R_{D_{0},X,A}\left(z\right)$ has a pole
of multiplicity $d_{j}$. Evidently, there exists an arbitrarily small
perturbation $\tilde{\mu}$ of $\mu_{0}$ for which this multiple
pole becomes a cluster of single poles, thereby changing the multiplicity
vector to some $D'\neq D_{0}$. While we address this problem in \prettyref{sec:collision-singularities}
via the bases of divided differences, in this section we consider
a ``multiplicity-restricted'' Prony problem. 
\begin{defn}
Let $\vec{x}=\left(x_{1},\dots,x_{s}\right)\in\CC^{s}$ and $D=\left(d_{1},\dots,d_{s}\right)$
with $d=\sum_{j=1}^{s}d_{j}$ be given. The\emph{ $d\times d$ confluent
Vandermonde} matrix is
\begin{equation}
\cvand=V\left(\vec{x},D\right)=\cvand\left(x_{1},d_{1},\dots,x_{s},d_{s}\right)=\left[\begin{array}{cccc}
\confvec[1][0] & \confvec[2][0] & \dotsc & \confvec[s][0]\\
\confvec[1][1] & \confvec[2][1] & \dotsc & \confvec[s][1]\\
 &  & \dotsc\\
\confvec[1][d-1] & \confvec[2][d-1] & \dotsc & \confvec[s][d-1]
\end{array}\right]\label{eq:confluent-vandermonde-def}
\end{equation}
where the symbol $\confvec$ denotes the following $1\times d_{j}$
row vector
\[
\confvec\isdef\left[\begin{array}{cccc}
x_{j}^{k}, & kx_{j}^{k-1}, & \dots & ,k\left(k-1\right)\cdots\left(k-d_{j}\right)x_{j}^{k-d_{j}+1}\end{array}\right].
\]
\end{defn}
\begin{prop}
The matrix $\cvand$ defines the linear part of the confluent Prony
system \eqref{eq:confluent equation_prony_system} in the standard
basis for $V_{w}$, namely, 
\begin{equation}
\cvand\left(x_{1},d_{1},\dots,x_{s},d_{s}\right)\begin{bmatrix}a_{1,0}\\
\vdots\\
a_{1,d_{1}-1}\\
\vdots\\
\\
a_{s,d_{s}-1}
\end{bmatrix}=\begin{bmatrix}m_{0}\\
m_{1}\\
\vdots\\
\\
\\
m_{d-1}
\end{bmatrix}.\label{eq:confluent-prony-linear-part}
\end{equation}
\end{prop}
\begin{defn}
Let $\PM\left(w_{0},g_{0}\right)=\mu_{0}\in\Sigma_{r_{0}}$ with $D\left(g_{0}\right)=D_{0}$
and $s\left(g_{0}\right)=s_{0}$. Let ${\cal P}_{D_{0}}$ denote the
following subbundle of ${\cal P}_{d}$ of dimension $s_{0}+r_{0}$:
\[
{\cal P}_{D_{0}}=\left\{ \left(w,g\right)\in{\cal P}_{d}:\quad D\left(g\right)=D_{0}\right\} .
\]

The multiplicity-restricted Prony mapping $\rpm:{\cal P}_{D_{0}}\to\CC^{s_{0}+r_{0}}$
is the composition 
\[
\rpm=\pi\circ\PM\restriction_{{\cal P}_{D_{0}}},
\]
where $\pi:\CC^{2d}\to\CC^{s_{0}+r_{0}}$ is the projection map on
the first $s_{0}+r_{0}$ coordinates.
\end{defn}
Inverting this $\rpm$ represents the solution of the confluent Prony
system \eqref{eq:confluent equation_prony_system} with fixed structure
$D_{0}$ from the first $k=0,1,\dots,s_{0}+r_{0}-1$ measurements.
\begin{thm}[\cite{batenkov2011accuracy}]
Let $\mu_{0}^{*}=\rpm\left(\left(w_{0},g_{0}\right)\right)\in\CC^{s_{0}+r_{0}}$
with the unperturbed solution $g_{0}=\sum_{j=1}^{s_{0}}\sum_{\ell=0}^{d_{j}-1}a_{j,\ell}\delta^{\left(\ell\right)}\left(x-\tau_{j}\right)$.
In a small neighborhood of $\left(w_{0},g_{0}\right)\in{\cal P}_{D_{0}}$,
the map $\rpm$ is invertible. Consequently, for small enough $\err$,
the multiplicity-restricted Prony problem with input data $\tilde{\mu}^{*}\in\CC^{r_{0}+s_{0}}$
satisfying $\|\tilde{\mu}^{*}-\mu_{0}^{*}\|\leq\err$ has a unique
solution. The error in this solution satisfies
\begin{eqnarray*}
\left|\Delta a_{j,\ell}\right| & \leq & \frac{2}{\ell!}\left(\frac{2}{\delta}\right)^{s_{0}+r_{0}}\left(\frac{1}{2}+\frac{s_{0}+r_{0}}{\delta}\right)^{d_{j}-\ell}\left(1+\frac{\left|a_{j,\ell-1}\right|}{\left|a_{j,d_{j}-1}\right|}\right)\err,\\
\left|\Delta\tau_{j}\right| & \leq & \frac{2}{d_{j}!}\left(\frac{2}{\delta}\right)^{s_{0}+r_{0}}\frac{1}{\left|a_{j,d_{j}-1}\right|}\err,
\end{eqnarray*}
where $\delta\isdef\min_{i\neq j}\left|\tau_{i}-\tau_{j}\right|$
(for consistency we take $a_{j,-1}=0$ in the above formula).\end{thm}
\begin{proof}
[Proof outline]The Jacobian of $\rpm$ can be easily computed, and
it turns out to be equal to the product
\[
{\cal J}_{\rpm}=V\left(\tau_{1},d_{1}+1,\dots,\tau_{s_{0}},d_{s_{0}}+1\right)\diag\left\{ E_{j}\right\} 
\]
where $V$ is the \emph{confluent Vandermonde matrix \eqref{eq:confluent-vandermonde-def}}
on the nodes $\left(\tau_{1},\dots,\tau_{s_{0}}\right)$, with multiplicity
vector
\[
\tilde{D}_{0}=\left(d_{1}+1,\dots,d_{s_{0}}+1\right),
\]
while $E$ is the $\left(d_{j}+1\right)\times\left(d_{j}+1\right)$
block
\[
E_{j}=\begin{bmatrix}1 & 0 & 0 & \cdots & 0\\
0 & 1 & 0 & \cdots & a_{j,0}\\
\vdots & \vdots & \vdots & \ \ddots & \vdots\\
0 & 0 & 0 & \cdots & a_{j,d_{j}-1}
\end{bmatrix}.
\]
Since $\mu_{0}\in\Sigma_{r}$, the highest order coefficients $a_{j,d_{j}-1}$
are nonzero. Furthermore, since all the $\tau_{j}$ are distinct,
the matrix $V$ is nonsingular. Local invertability follows. To estimate
the norm of the inverse, use bounds from \cite{2012arXiv1212.0172B}.\end{proof}
\begin{rem}
Note that as two nodes collide ($\delta\to0$), the inversion of the
multiplicity-restricted Prony mapping $\rpm$ becomes ill-conditioned
proportionally to $\delta^{-\left(s_{0}+r_{0}\right)}$.
\end{rem}
Let us stress that we are not aware of any general method of inverting
$\rpm$, i.e. solving the multiplicity-restricted confluent Prony
problem with the smallest possible number of measurements. As we demonstrate
in \cite{batFullFourier}, such a method exists for a very special
case of a single point, i.e. $s=1$.

\section{\label{sec:rank-restriction}Rank-restricted Prony problem}

Recall that the Prony problem consists in inverting the Prony mapping
$\PM:{\cal P}_{d}\rightarrow{\cal T}_{d}$. So, given $\mu=(m_{0},\dots,m_{2d-1})\in{\cal T}_{d}$
we are looking for $(w,g)\in{\cal P}_{d}$ such that $m_{k}(g)=\int x^{k}g(x)dx=m_{k}$,
with $k=0,1,\dots,2d-1$. If $\mu\in\Sigma_{r}$ with $r<d$, then
in fact any neighborhood of $\mu$ will contain points from the non-solvability
set $\Sigma'$. Indeed, consider the following example.
\begin{example}
\label{ex:unsolvability-for-small-perturbation}Slightly modifying
the construction of \prettyref{ex:nonsolv-ex2}, consider $\mu_{\ell_{1},\ell_{2},\e}\in\CC^{2d}$
with all the $m_{k}=0$ besides $m_{\ell_{1}}=1$ and $m_{\ell_{2}}=\e$,
such that $\ell_{2}>\ell_{1}+d-1$. For example, if $d=5$ and $\ell_{1}=2,\;\ell_{2}=8$,
the corresponding matrix is
\[
\M_{5}^{\left(2,8,\e\right)}=\begin{bmatrix}0 & 0 & 1 & 0 & 0 & 0\\
0 & 1 & 0 & 0 & 0 & 0\\
1 & 0 & 0 & 0 & 0 & 0\\
0 & 0 & 0 & 0 & 0 & \e\\
0 & 0 & 0 & 0 & \e & 0
\end{bmatrix}.
\]
For $\e=0$ the Prony problem is solvable, while for any small perturbation
$\e\neq0$ it becomes unsolvable. However, if we restrict the whole
problem just to $d=3$, it remains solvable for any small perturbation
of the input.
\end{example}
\global\long\def\rrpm{{\cal \PM}_{r}^{*}}

We therefore propose to consider the \emph{rank-restricted Prony problem}
analogous to the construction of \prettyref{sec:multiplicity-restricted},
but instead of fixing the multiplicity $D\left(g\right)$ we now fix
the rank $r$ (recall \prettyref{def:mult-vec-signal}).
\begin{defn}
Denote by ${\cal P}_{r}$ the following vector bundle:
\[
{\cal P}_{r}=\left\{ \left(w,g\right):\quad w\in\CC^{r},\; g\in V_{w}\right\} ,
\]
where $V_{w}$ is defined exactly as in \prettyref{def:bundles},
replacing $d$ with $r$.
\end{defn}
Likewise, we define the Stieltjes bundle of order $r$ as follows.
\begin{defn}
Denote by ${\cal S}_{r}$ the following vector bundle:
\[
{\cal S}_{r}=\left\{ \left(w,\g\right):\qquad w\in\CC^{r},\;\g\in W_{w}\right\} ,
\]
where $W_{w}$ is defined exactly as in \prettyref{def:bundles},
replacing $d$ with $r$.
\end{defn}
The Stieltjes mapping acts naturally as a map $\SM:{\cal P}_{r}\to{\cal S}_{r}$
with exactly the same definition as \prettyref{def:stieltjes-mapping}.

The restricted Taylor mapping $\TM_{r}:{\cal S}_{r}\to\CC^{2r}$ is,
as before, given by the truncated development at infinity to the first
$2r$ Taylor coefficients.
\begin{defn}
Let $\pi:\CC^{2d}\to\CC^{2r}$ denote the projection operator onto
the first $2r$ coordinates. Denote $\Sigma_{r}^{*}\isdef\pi\left(\Sigma_{r}\right)$.
The rank-restricted Prony mapping $\rrpm:{\cal P}_{r}\to\Sigma_{r}^{*}$
is given by by
\[
\rrpm\left(\left(w,g\right)\right)=\left(m_{0},\dots,m_{2r-1}\right),\qquad m_{k}=m_{k}\left(g\right)=\int x^{k}g\left(x\right)\dd x.
\]
\end{defn}
\begin{rem}
${\cal P}_{r}$ can be embedded in ${\cal P}_{d}$, for example by
the map $\Xi_{r}:{\cal P}_{r}\to{\cal P}_{d}$ 
\[
\Xi_{r}:\;\left(w,g\right)\in{\cal P}_{r}\longmapsto\left(w',g'\right)\in{\cal P}_{d}:\qquad w'=\left(x_{1},\dots,x_{r},\underbrace{0,\dots0}_{\times\left(d-r\right)}\right),\; g'=g.
\]
With this definition,  $\rrpm$ can be represented also as the composition
\[
\rrpm=\pi\circ\PM\circ\Xi_{r}.
\]
\end{rem}
\begin{prop}
The rank-restricted Prony mapping satisfies
\[
\rrpm=\TM_{r}\circ\SM.
\]

\end{prop}
Inverting $\rrpm$ represents the solution of the rank-restricted
Prony problem. Unlike in the multiplicity-restricted setting of \prettyref{sec:multiplicity-restricted},
here we allow two or more nodes to collide (thereby changing the multiplicty
vector $D\left(g\right)$ of the solution).

The basic fact which makes this formulation useful is the following
result.
\begin{thm}
\label{thm:rank-restriction-is-good}Let $\mu_{0}^{*}\in\Sigma_{r}^{*}$.
Then in a small neighborhood of $\mu_{0}^{*}\in\CC^{2r}$, the Taylor
mapping $\TM_{r}$ is continuously invertible.\end{thm}
\begin{proof}
This is a direct consequence of the solution method to the Padé approximation
problem described in \prettyref{app:proof-pade-solvability}. Indeed,
if the rank of $\M_{r}$ is full, then it remains so in a small neighborhood
of the \emph{entire space $\CC^{2r}$. }Therefore, the system \eqref{eq:double-star}
remains continuously invertible, producing the coefficients of the
denominator $Q\left(z\right)$. Consequently, the right-hand side
of \eqref{eq:star} depends continuously on the moment vector $\mu^{*}=\left(m_{0},\dots,m_{2r-1}\right)\in\CC^{2r}$.
Again, since the rank always remains full, the polynomials $P\left(z\right)$
and $Q\left(z\right)$ cannot have common roots, and thereby the solution
$R=\frac{P}{Q}=\TM_{r}^{-1}\left(\mu^{*}\right)$ depends continuously
on $\mu^{*}$ (in the topology of the space of rational functions).
\end{proof}
In the next section, we consider the remaining problem: how to invert
$\SM$ in this setting.

\section{\label{sec:collision-singularities}Collision singularities and bases
of finite differences}

\global\long\def\DD{\mathord{\kern0.43em  \vrule width.6ptheight5.6ptdepth-.28pt\kern-0.43em  \Delta}}

\subsection{Introduction}

Collision singularities occur in Prony systems as some of the nodes
$x_{i}$ in the signal $F(x)=\sum_{i=1}^{d}a_{i}\delta(x-x_{i})$
approach one another. This happens for $\mu$ near the discriminant
stratum $\Delta\subset\CC^{2d}$ consisting of those $(m_{0},\dots,m_{2d-1})$
for which some of the coordinates $\left\{ x_{j}\right\} $ in the
solution collide, i.e. the function $R_{D,X,A}\left(z\right)$ has
multiple poles (or, nontrivial multiplicity vector $D$). As we shall
see below, typically, as $\mu$ approaches $\mu_{0}\in\Delta$, i.e.
some of the nodes $x_{i}$ collide, the corresponding coefficients
$a_{i}$ tend to infinity. Notice, that all the moments $m_{k}=m_{k}(F)$
remain bounded. This behavior creates serious difficulties in solving
``near-colliding'' Prony systems, both in theoretical and practical
settings. Especially demanding problems arise in the presence of noise.
The problem of improvement of resolution in reconstruction of colliding
nodes from noisy measurements appears in a wide range of applications.
It is usually called a ``super-resolution problem'' and a lot of
recent publications are devoted to its investigation in various mathematical
and applied settings. See \cite{candes2012towards} and references
therein for a very partial sample.

Here we continue our study of collision singularities in Prony systems,
started in \cite{yom2009Singularities}. Our approach uses bases of
finite differences in the Prony space ${\cal P}_{r}$ in order to
``resolve'' the linear part of collision singularities. In these
bases the coefficients do not blow up any more, even as some of the
nodes collide.
\begin{example}
\label{ex:colliding}Let $r=2$, and consider the signal $F=a_{1}\delta\left(x-x_{1}\right)+a_{2}\delta\left(x-x_{2}\right)$
with 
\begin{eqnarray*}
x_{1} & = & t,\; x_{2}=t+\epsilon,\\
a_{1} & = & -\epsilon^{-1},\; a_{2}=\epsilon^{-1}.
\end{eqnarray*}
The corresponding Prony system is
\[
\left(a_{1}x_{1}^{k}+a_{2}x_{2}^{k}=\right)m_{k}=kt^{k-1}+\underbrace{\sum_{j=2}^{k}{k \choose j}t^{k-j}\e^{j-1}}_{\isdef\rho_{k}\left(t,\e\right)},\qquad k=0,1,2,3.
\]
As $\epsilon\to0$, the Prony system as above becomes ill-conditioned
and the coefficients $\left\{ a_{j}\right\} $ blow up, while the
measurements remain bounded. Note that
\[
\M_{2}=\begin{bmatrix}0 & 1 & 2t+\rho_{2}\left(t,\e\right)\\
1 & 2t+\rho_{2}\left(t,\e\right) & 3t^{2}+\rho_{3}\left(t,\e\right)
\end{bmatrix},
\]
therefore $\rank\M_{2}=2$ and $\left|M_{2}\right|=1\neq0$, i.e.
the Prony problem with input $\left(m_{0},\dots,m_{3}\right)$ remains
solvable for all $\e$. However, the standard basis $\left\{ \delta\left(x-x_{1}\right),\;\delta\left(x-x_{2}\right)\right\} $
degenerates, and in the limit it is no more a basis. If we represent
the solution
\[
F_{\e}\left(x\right)=-\frac{1}{\e}\delta\left(x-t\right)+\frac{1}{\e}\delta\left(x-t-\e\right)
\]
in the basis
\begin{eqnarray*}
\Delta_{1}\left(x_{1},x_{2}\right) & = & \delta\left(x-x_{1}\right),\\
\Delta_{2}\left(x_{1},x_{2}\right) & = & \frac{1}{x_{1}-x_{2}}\delta\left(x-x_{1}\right)+\frac{1}{x_{2}-x_{1}}\delta\left(x-x_{2}\right),
\end{eqnarray*}
then we have
\[
F_{\e}\left(x\right)=1\cdot\Delta_{2}\left(t,t+\e\right),
\]
i.e. the coefficients in this new basis are just $\left\{ b_{1}=0,\; b_{1}=1\right\} $.
As $\e\to0$, in fact we have
\[
\Delta_{2}\left(t,t+\e\right)\to\delta'\left(x-t\right),
\]
where the convergence is in the topology of the bundle ${\cal P}_{r}$.
\end{example}

Our goal in this section is to generalize the construction of \prettyref{ex:colliding}
and \cite{yom2009Singularities} to handle the general case of colliding
configurations.

\subsection{Divided finite differences}

For modern treatment of divided differences, see e.g. \cite{de2005divided,kahan1999symbolic,milne1933calculus}.
We follow \cite{de2005divided} and adopt what has become by now the
standard definition. 
\begin{defn}
Let an arbitrary sequence of points $w=\left(x_{1},x_{2},\dots,\right)$
be given (repetitions are allowed). The \emph{(n-1})-st \emph{divided
difference }$\DD^{n-1}\left(w\right):\Pi\to\CC$ is the linear functional
on the space $\Pi$ of polynomials, associating to each $p\in\Pi$
its (uniquely defined) $n$-th coefficient in the Newton form \emph{
\begin{equation}
p\left(x\right)=\sum_{j=1}^{\infty}\left\{ \DD^{j-1}\left(x_{1},\dots,x_{j}\right)p\right\} \cdot q_{j-1,w}\left(x\right),\qquad q_{i,w}\left(x\right)\isdef\prod_{j=1}^{i}\left(x-x_{j}\right).\label{eq:newton-expansion}
\end{equation}
}
\end{defn}
It turns out that this definition can be extended to all sufficiently
smooth functions for which the interpolation problem is well-defined.
\begin{defn}[\cite{de2005divided}]
For any smooth enough function $f$, defined at least on $x_{1},\dots,x_{n}$,
the divided finite difference $\DD^{n-1}\left(x_{1},\dots,x_{n}\right)f$
is the $n$-th coefficient in the Newton form \eqref{eq:newton-expansion}
of the Hermite interpolation polynomial $P_{n}$, which agrees with
$f$ and its derivatives of appropriate order on $x_{1},\dots,x_{n}:$
\begin{equation}
f^{\left(\ell\right)}\left(x_{j}\right)=P_{n}^{\left(\ell\right)}\left(x_{j}\right):\qquad1\leqslant j\leqslant n,\;0\leqslant\ell<d_{j}\isdef\#\left\{ i:\quad x_{i}=x_{j}\right\} .\label{eq:hermite-conditions-dd}
\end{equation}

\end{defn}
Therefore, each divided difference can be naturally associated with
an element of the Prony space (see \prettyref{enu:dd-as-functionals}
in \prettyref{prop:dd-properties} and \prettyref{def:dd-as-elements-of-vw}
below for an accurate statement).

Let us now summarize relevant properties of the functional $\DD$
which we shall use later on.
\begin{prop}
\label{prop:dd-properties}For $w=\left(x_{1},\dots,x_{n}\right)\in\CC^{n}$,
let $s\left(w\right),\; T\left(w\right)$ and $D\left(w\right)$ be
defined according to \prettyref{def:config}. Let $q_{n,w}\left(z\right)=\prod_{j=1}^{s}\left(z-\tau_{j}\right)^{d_{j}}$
be defined as in \eqref{eq:newton-expansion}.
\begin{enumerate}
\item \label{enu:dd-ordering}The functional $\DD^{n-1}\left(x_{1},\dots,x_{n}\right)$
is a symmetric function of its arguments, i.e. it depends only on
the set $\left\{ x_{1},\dots,x_{n}\right\} $ but not on its ordering.
\item $\DD^{n-1}\left(x_{1},\dots,x_{n}\right)$ is a continuous function
of the vector $\left(x_{1},\dots,x_{n}\right)$. In particular, for
any test function $f$
\[
\lim_{\left(x_{1},\dots,x_{n}\right)\to\left(t_{1},\dots,t_{n}\right)}\DD^{n-1}\left(x_{1},\dots,x_{n}\right)f=\DD^{n-1}\left(t_{1},\dots,t_{n}\right)f.
\]

\item $\DD$ may be computed by the recursive rule
\begin{equation}
\DD^{n-1}\left(x_{1},\dots,x_{n}\right)f=\begin{cases}
\frac{\DD^{n-2}\left(x_{2},\dots,x_{n}\right)f-\DD^{n-2}\left(x_{1},\dots,x_{n-1}\right)f}{x_{n}-x_{1}} & x_{1}\neq x_{n},\\
\left\{ \frac{\dd}{\dd\xi}\DD^{n-2}\left(\xi,x_{2},\dots,x_{n-1}\right)f\right\} |_{\xi=x_{n}}, & x_{1}=x_{n},
\end{cases}\label{eq:dd-recursive-rule}
\end{equation}
where $\DD^{0}\left(x_{1}\right)f=f\left(x_{1}\right).$
\item Let $f_{z}\left(x\right)=\left(z-x\right)^{-1}$. Then for all $z\notin\left\{ x_{1},\dots,x_{n}\right\} $
\begin{equation}
\DD^{n-1}\left(x_{1},\dots,x_{n}\right)f_{z}=\frac{1}{q_{n,w}\left(z\right)}.\label{eq:dd-of-fraction}
\end{equation}

\item \label{enu:dd-as-functionals}By \eqref{eq:hermite-conditions-dd},
$\DD^{n-1}\left(x_{1},\dots,x_{n}\right)$ is a linear combination
of the functionals
\[
\delta^{\left(\ell\right)}\left(x-\tau_{j}\right),\qquad1\leqslant j\leqslant s,\;0\leqslant\ell<d_{j}.
\]
In fact, using \eqref{eq:dd-of-fraction} we obtain the Chakalov's
expansion (see \cite{de2005divided})
\begin{equation}
\DD^{n-1}\left(x_{1},\dots,x_{n}\right)=\sum_{j=1}^{s}\sum_{\ell=0}^{d_{j}-1}a_{j,\ell}\delta^{\left(\ell\right)}\left(x-\tau_{j}\right),\label{eq:chakalov}
\end{equation}
where the coefficients $\left\{ a_{j,\ell}\right\} $ are defined
by the partial fraction decomposition%
\footnote{The coefficients $\left\{ a_{j,\ell}\right\} $ may be readily obtained
by the Cauchy residue formula
\[
a_{j,\ell}=\frac{1}{\left(d_{j}-1-\ell\right)!}\lim_{z\to\tau_{j}}\left(\frac{\dd}{\dd z}\right)^{d_{j}-1-\ell}\left\{ \frac{\left(z-\tau_{j}\right)^{\ell+1}}{q_{n,w}\left(z\right)}\right\} .
\]
}
\begin{equation}
\frac{1}{q_{n,w}\left(z\right)}=\sum_{j=1}^{s}\sum_{\ell=0}^{d_{j}-1}\frac{\ell!a_{j,\ell}}{\left(z-\tau_{j}\right)^{\ell+1}}.\label{eq:chakalov-coeffs}
\end{equation}

\item By \eqref{eq:chakalov} and \eqref{eq:chakalov-coeffs}
\begin{equation}
\DD^{n-1}\left(\underbrace{t,\dots,t}_{\times n}\right)=\frac{1}{\left(n-1\right)!}\delta^{\left(n-1\right)}\left(x-t\right).\label{eq:dd-same-nodes-derivative}
\end{equation}

\item Popoviciu's refinement lemma \cite[Proposition 23]{de2005divided}:
for every index subsequence
\[
1\leqslant\sigma\left(1\right)<\sigma\left(2\right)<\dots<\sigma\left(k\right)\leqslant n,
\]
there exist coefficients $\alpha\left(j\right)$ such that
\begin{equation}
\DD^{k-1}\left(x_{\sigma\left(1\right)},\dots,x_{\sigma\left(k\right)}\right)=\sum_{j=\sigma\left(1\right)-1}^{\sigma\left(k\right)-k}\alpha\left(j\right)\DD^{k-1}\left(x_{j+1},x_{j+2},\dots,x_{j+k}\right).\label{eq:refinement}
\end{equation}

\end{enumerate}
\end{prop}
Based on the above, we may now identify $\DD$ with elements of the
bundle ${\cal P}_{r}$.
\begin{defn}
\label{def:dd-as-elements-of-vw}Let $w=\left(x_{1},\dots,x_{r}\right)\in\CC^{r}$,
and $X=\left\{ n_{1},n_{2},\dots,n_{\alpha}\right\} \subseteq\left\{ 1,2,\dots,r\right\} $
of size $\left|X\right|=\alpha$ be given. Let the elements of $X$
be enumerated in increasing order, i.e.
\[
1\leqslant n_{1}<n_{2}<\dots<n_{\alpha}\leqslant r.
\]
Denote by $w_{X}$ the vector
\[
w_{X}\isdef\left(x_{n_{1}},x_{n_{2}},\dots,x_{n_{\alpha}}\right)\in\CC^{\alpha}.
\]
Then we denote 
\[
\Delta_{X}\left(w\right)\isdef\DD^{\alpha-1}\left(w_{X}\right).
\]

\end{defn}
We immediately obtain the following result.
\begin{lem}
\label{lem:dd-as-our-guys}For all $w\in\CC^{r}$ and $X\subseteq\left\{ 1,2,\dots,r\right\} $,
we have $\Delta_{X}\left(w\right)\in V_{w}$. Moreover, letting $\alpha=\left|X\right|$
we have
\begin{equation}
\SM\left(\Delta_{X}\left(w\right)\right)=\DD^{\alpha-1}\left(w_{X}\right)\frac{1}{z-x}=\frac{1}{q_{\alpha,w_{X}}\left(z\right)}.\label{eq:sm-dd}
\end{equation}
Finally, $\left(w,\Delta_{X}\left(w\right)\right)$ is a continuous
section of ${\cal P}_{r}$.
\end{lem}

\subsection{\label{sub:dd-basis}Constructing a basis}

The following result is well-known, see e.g. \cite[Proposition 35]{de2005divided}.
\begin{thm}
\label{thm:basis-all}Denote $N_{j}=\left\{ 1,2,\dots,j\right\} $
for $j=1,2,\dots,r$. Then for every $w\in\CC^{r}$, the collection
\[
\left\{ \Delta_{N_{j}}\left(w\right)\right\} _{j=1}^{r}
\]
is a basis for $V_{w}$.
\end{thm}
There are various proofs of this statement. Below we show how to construct
sets which do not necessarily remain basis for all $w\in\CC^{r}$,
but only for $w$ in a small neighborhood of a given $w_{0}\in\CC^{r}.$
\prettyref{thm:basis-all} will then follow as a special case of this
construction.

Informally, if two coordinates $x_{i}$ and $x_{j}$ can collide,
then it is necessary to allow them to be glued by some element of
the basis, i.e. we will need $\Delta_{X}\left(w\right)$ where $i,j\in X$
(in \prettyref{thm:basis-all} all coordinates might be eventually
glued into a single point because $w$ is unrestricted.) In order
to make this statement formal, let us introduce a notion of \emph{configuration},
which is essentially a partition of the set of indices.
\begin{defn}
A \emph{configuration} ${\cal C}$ is a partition of the set $N_{r}=\left\{ 1,2,\dots,r\right\} $
into $s=s\left({\cal C}\right)$ disjoint nonempty subsets
\[
\sqcup_{i=1}^{s}X_{i}=N_{r},\qquad\left|X_{i}\right|=d_{i}>0.
\]
The multiplicity vector of ${\cal C}$ is
\[
T\left({\cal C}\right)=\left(d_{1},\dots,d_{s}\right).
\]

\end{defn}
Every configuration defines a continuous family of divided differences
as follows.
\begin{defn}
Let a configuration ${\cal C}=\left\{ X_{j}\right\} _{j=1}^{s\left({\cal C}\right)}$
. Enumerate each $X_{j}$ in increasing order of its elements
\[
X_{j}=\left\{ n_{1}^{j}<n_{2}^{j}<\dots n_{d_{j}}^{j}\right\} 
\]
and denote for every $m=1,2,\dots,d_{j}$
\[
X_{j,m}\isdef\left\{ n_{k}^{j}:\; k=1,2,\dots,m\right\} .
\]
For every $w\in\CC^{r}$, the collection ${\cal B_{C}}\left(w\right)\subset V_{w}$
is defined as follows:
\[
{\cal B_{C}}\left(w\right)\isdef\left\{ \Delta_{X_{j,m}}\left(w\right)\right\} _{j=1,\dots,s\left({\cal C}\right)}^{m=1,\dots,d_{j}}.
\]

\end{defn}
Now we formally define when a partition is ``good'' with respect
to a point $w\in\CC^{r}$.
\begin{defn}
The point $w=\left(x_{1},\dots,x_{r}\right)\in\CC^{r}$ is \emph{subordinated}
to the configuration ${\cal C}=\left\{ X_{j}\right\} _{j=1}^{s\left({\cal C}\right)}$
if whenever $x_{k}=x_{\ell}$ for a pair of indices $k\neq\ell$,
then necessarily $k,\ell\in X_{j}$ for some $X_{j}$.
\end{defn}
Now we are ready to formulate the main result of this section.
\begin{thm}
\label{thm:basis-conf}For a given $w_{0}\in\CC^{r}$ and a configuration
${\cal C}$, the collection ${\cal B_{C}}\left(w_{0}\right)$ is a
basis for $V_{w_{0}}$ if and only if $w_{0}$ is subordinated to
${\cal C}$. In this case, ${\cal B_{C}}\left(w\right)$ is a continuous
family of bases for $V_{w}$ in a sufficiently small neighborhood
of $w_{0}$.
\end{thm}
Let us first make a technical computation.
\begin{lem}
\label{lem:belongs-to-span}For a configuration ${\cal C}$ and a
point $w\in\CC^{r}$, consider for every fixed $j=1,\dots,s\left({\cal C}\right)$
the set 
\begin{equation}
S_{j}\isdef\left\{ \Delta_{X_{j,m}}\left(w\right)\right\} _{m=1}^{d_{j}}.\label{eq:subbasis}
\end{equation}

\begin{enumerate}
\item Define for any pair of indices $1\leqslant k\leqslant\ell\leqslant d_{j}$
the index set
\[
X_{j,k:\ell}\isdef\left\{ n_{k}^{j}<n_{k+1}^{j}<\dots<n_{\ell}^{j}\right\} \subseteq X_{j}=X_{j,1:d_{j}}=X_{j,d_{j}}.
\]
Then 
\[
\Delta_{X_{j,k:\ell}}\left(w\right)\in\spann S_{j}.
\]

\item For an arbitrary subset $Y\subseteq X_{j}$ (and not necessarily containing
segments of consecutive indices), we also have
\[
\Delta_{Y}\left(w\right)\in\spann S_{j}.
\]

\end{enumerate}
\end{lem}
\begin{proof}
For clarity, we denote $y_{i}=x_{n_{i}^{j}}$ and $\left[k:\ell\right]=\Delta_{X_{j,k:\ell}}\left(w\right)$.
By \eqref{eq:dd-recursive-rule} we have in all cases (including repeated
nodes)
\begin{equation}
\left(y_{\ell}-y_{k}\right)\left[k:\ell\right]=\left[k+1:\ell\right]-\left[k:\ell-1\right].\label{eq:recursive-2}
\end{equation}

The proof of the first statement is by backward induction on $n=\ell-k$.
We start from $n=d_{j}$, and obviously $\left[1:d_{j}\right]\in S_{j}$.
In addition, by definition of $S_{j}$ we have $\left[1:m\right]\in S_{j}$
for all $m=1,\dots,d_{j}$. Therefore, in order to obtain all $\left[k:\ell\right]$
with $\ell-k=n-1$, we apply \eqref{eq:recursive-2} several times
as follows. 

\begin{eqnarray*}
\left[2:n\right] & = & \left(y_{n}-y_{1}\right)\left[1:n\right]+\left[1:n-1\right]\\
\left[3:n+1\right] & = & \left(y_{n+1}-y_{2}\right)\underleftrightarrow{\left[2:n+1\right]}+\underline{\left[2:n\right]}\\
 & \dots\\
\left[d_{j}-n+2:d_{j}\right] & = & \left(y_{d_{j}}-y_{d_{j}-n+1}\right)\underleftrightarrow{\left[d_{j}-n+1:d_{j}\right]}+\underline{\left[d_{j}-n+1:d_{j}-1\right]}
\end{eqnarray*}
Here the symbol $\underline{\cdots}$ under a term means that the
term is taken directly from the previous line, while $\underleftrightarrow{\cdots}$
indicates that the induction hypothesis is used. In the end, the left-hand
side terms are shown to belong to $\spann S_{j}$.

In order to prove the second statement, we employ the first statement,
\eqref{eq:refinement} and \prettyref{prop:dd-properties}, \prettyref{enu:dd-ordering}.
\end{proof}
\begin{minipage}[t]{1\columnwidth}%
\end{minipage}
\begin{proof}[Proof of \prettyref{thm:basis-conf}]
In one direction, assume that $w_{0}=\left(x_{1},\dots,x_{r}\right)$
is subordinated to ${\cal C}$. It is sufficient to show that every
element of the standard basis \eqref{eq:standard-basis-prony} belongs
to $\spann\left\{ {\cal B_{C}}\left(w_{0}\right)\right\} $.

Let $\tau_{j}\in T\left(w_{0}\right)$, let $d_{j}$ be the corresponding
multiplicity, and let $Y_{j}\subseteq N_{r}$ denote the index set
of size $d_{j}$ 
\[
Y_{j}\isdef\left\{ i:\quad x_{i}=\tau_{j}\right\} .
\]
 By the definition of subordination, there exists an element in the
partition of ${\cal C}$, say $X_{k}$, for which $Y_{j}\subseteq X_{k}$.
By \prettyref{lem:belongs-to-span} we conclude that for all subsets
$Z\subseteq Y_{j}$,
\[
\Delta_{Z}\left(w_{0}\right)\in\spann\left\{ \Delta_{X_{k,m}}\left(w_{0}\right)\right\} _{m=1}^{\left|X_{k}\right|}\subseteq\spann\left\{ {\cal B_{C}}\left(w_{0}\right)\right\} .
\]
By \eqref{eq:dd-same-nodes-derivative}, $\Delta_{Z}\left(w_{0}\right)$
is nothing else but
\[
\Delta_{Z}\left(w_{0}\right)=\DD^{\left|Z\right|-1}\left(\underbrace{\tau_{j},\dots,\tau_{j}}_{\times\left|Z\right|}\right)=\frac{1}{\left(\left|Z\right|-1\right)!}\delta^{\left(\left|Z\right|-1\right)}\left(x-\tau_{j}\right).
\]
This completes the proof of the necessity. In the other direction,
assume by contradiction that $x_{k}=x_{\ell}=\tau$ but nevertheless
there exist two distinct elements of the partition ${\cal C}$, say
$X_{\alpha}$ and $X_{\beta}$ such that $k\in X_{\alpha}$ and $\ell\in X_{\beta}.$
Let the sets $\left\{ S_{j}\right\} _{j=1}^{s\left({\cal C}\right)}$
be defined by \prettyref{eq:subbasis}. Again, by \prettyref{lem:belongs-to-span}
and \eqref{eq:dd-same-nodes-derivative} we conclude that
\[
\delta\left(x-\tau\right)\in\spann S_{\alpha}\cap\spann S_{\beta}.
\]
But notice that ${\cal B_{C}}\left(w_{0}\right)=\bigcup_{j=1}^{s\left({\cal C}\right)}S_{j}$
and $\sum_{j=1}^{s}\left|S_{j}\right|=d$, therefore by counting dimensions
we conclude that
\[
\dim\spann\left\{ {\cal B_{C}}\left(w_{0}\right)\right\} <d,
\]
in contradiction to the assumption that ${\cal B_{C}}\left(w_{0}\right)$
is a basis.

Finally, one can evidently choose a sufficiently small neighborhood
$U\subset\CC^{r}$ of $w_{0}$ such that for all $w\in U$, no new
collisions are introduced, i.e. $w$ is still subordinated to ${\cal C}$.
The continuity argument (\prettyref{lem:dd-as-our-guys}) finishes
the proof.\end{proof}
\begin{rem}
Another possible method of proof is to consider the algebra of elementary
fractions in the Stieltjes space ${\cal S}_{r}$, and use the correspondence
\eqref{eq:sm-dd}.
\end{rem}
As we mentioned, \prettyref{thm:basis-all} follows as a corollary
of \prettyref{thm:basis-conf} for the configuration ${\cal C}$ consisting
of a single partition set $N_{r}$.

\subsection{Resolution of collision singularities}

Let $\mu_{0}^{*}\in\Sigma_{r}^{*}\subset\CC^{2r}$ be given, and let
$\left(w_{0},g_{0}\right)\in{\cal P}_{r}$ be a solution to the (rank-restricted)
Prony problem. The point $w_{0}$ is uniquely defined up to a permutation
of the coordinates, so we just fix a particular permutation. Let $T\left(w_{0}\right)=\left(\tau_{1},\dots,\tau_{s}\right)$.

Our goal is to solve the rank-restricted Prony problem for every input
$\mu^{*}\in\CC^{2r}$ in a small neighborhood of $\mu_{0}^{*}$. According
to \prettyref{thm:rank-restriction-is-good}, this amounts to a continuous
representation of the solution $R_{\mu^{*}}\left(z\right)=\frac{P_{\mu^{*}}\left(z\right)}{Q_{\mu^{*}}\left(z\right)}=\TM_{r}^{-1}\left(\mu^{*}\right)$
to the corresponding diagonal Padé approximation problem as an element
of the bundle ${\cal P}_{r}$.

Define $\delta=\min_{i\neq j}\left|\tau_{i}-\tau_{j}\right|$ to be
the ``separation distance'' between the clusters. Since the roots
of $Q_{\mu^{*}}$ depend continuously on $\mu^{*}$ and the degree
of $Q_{\mu^{*}}$ does not drop, we can choose some $\mu_{1}^{*}$
sufficiently close to $\mu_{0}^{*}$, for which
\begin{enumerate}
\item all the roots of $Q_{\mu_{1}^{*}}\left(z\right)$ are distinct, and 
\item these roots can be grouped into $s$ clusters, such that each of the
elements of the $j$-th cluster is at most $\delta/3$ away from $\tau_{j}$.
\end{enumerate}
Enumerate the roots of $Q_{\mu_{1}^{*}}$ within each cluster in an
arbitrary manner. This choice enables us to define locally (in a neighborhood
of $\mu_{1}^{*}$) $r$ algebraic functions $x_{1}\left(\mu^{*}\right),\dots,x_{r}\left(\mu^{*}\right)$,
satisfying
\[
Q_{\mu^{*}}\left(z\right)=\prod_{j=1}^{s}\left(z-x_{j}\left(\mu^{*}\right)\right).
\]
Then we extend these functions by analytic continuation according
to the above formula into the entire neighborhood of $\mu_{0}^{*}$.
Consequently,
\[
w\left(\mu^{*}\right)\isdef\left(x_{1}\left(\mu^{*}\right),\dots,x_{r}\left(\mu^{*}\right)\right)
\]
is a continuous (multivalued) algebraic function in a neighborhood
of $\mu_{0}^{*}$, satisfying
\[
w\left(\mu_{0}^{*}\right)=w_{0}.
\]

After this ``pre-processing'' step, we can solve the rank-restricted
Prony problem in this neighborhood of $\mu_{0}^{*}$, as follows.

\begin{algorithm}[H]
\noindent \begin{raggedright}
Let $\mu_{0}^{*}\in\Sigma_{r}^{*}\subset\CC^{2r}$ be given, and let
$\left(w_{0},g_{0}\right)\in{\cal P}_{r}$ be a solution to the (rank-restricted)
Prony problem. Let $w_{0}$ be subordinated to some configuration
${\cal C}$.
\par\end{raggedright}

\noindent \begin{raggedright}
The input to the problem is a measurement vector $\mu^{*}=\left(m_{0},\dots,m_{2r-1}\right)\in\CC^{2r}$,
which is in a small neighborhood of $ $$\mu_{0}^{*}$.
\par\end{raggedright}
\begin{enumerate}
\item \noindent \begin{raggedright}
Construct the function $w=w\left(\mu^{*}\right)$ as described above.
\par\end{raggedright}
\item \noindent \begin{raggedright}
Build the basis ${\cal B_{C}}\left(w\right)=\left\{ \Delta_{X_{j,\ell}}\left(w\right)\right\} _{j=1,\dots,s\left({\cal C}\right)}^{\ell=1,\dots,d_{j}}$
for $V_{w}$.
\par\end{raggedright}
\item \noindent \begin{raggedright}
Find the coefficients $\left\{ \beta_{j,\ell}\right\} _{j=1,\dots,s\left({\cal C}\right)}^{\ell=1,\dots,d_{j}}$
such that 
\[
\SM\left(\sum_{j,\ell}\beta_{j,\ell}\Delta_{X_{j,\ell}}\left(w\right)\right)=R\left(z\right),
\]
by solving the linear system
\begin{equation}
\underbrace{\sum_{j,\ell}\beta_{j,\ell}\left(w\right)\Delta_{X_{j,\ell}}\left(w\right)}_{=g\left(w\right)}\left(x^{k}\right)=m_{k}\left(=\int x^{k}g\left(w\right)\left(x\right)\dd x\right),\qquad k=0,1,\dots,2r-1.\label{eq:linear-system-dd}
\end{equation}

\par\end{raggedright}
\end{enumerate}
\label{alg:prony-dd}\caption{Solving rank-restricted Prony problem with collisions.}
\end{algorithm}

\begin{thm}
\label{thm:f.d.prony.1}The coordinates $\left\{ \beta_{j,\ell}\right\} $
of the solution to the rank-restricted Prony problem, given by \prettyref{alg:prony-dd},
are (multivalued) algebraic functions, continuous in a neighborhood
of the point $\mu_{0}^{*}$ .\end{thm}
\begin{proof}
Since the divided differences $\Delta_{j,\ell}\left(w\right)$ are
continuous in $w$, then clearly for each $k=0,1,\dots,2r-1$ the
functions
\[
\nu_{j,\ell,k}\left(w\right)=\Delta_{j,\ell}\left(w\right)\left(x^{k}\right)=\DD^{\ell-1}\left(w_{X_{j,\ell}}\right)\left(x^{k}\right)
\]
are continuous%
\footnote{In fact, $\nu_{j,\ell,k}\left(w\right)$ are symmetric polynomials
in some of the coordinates of $w$.%
} in $w$, and hence continuous, as multivalued functions, in a neighborhood
of $\mu_{0}^{*}$. Since ${\cal B_{C}}\left(w\left(\mu^{*}\right)\right)$
remains a basis in a (possibly smaller) neighborhood of $\mu_{0}^{*}$,
the system \eqref{eq:linear-system-dd}, taking the form
\[
\sum_{j,\ell}\nu_{j,\ell,k}\left(w\right)\beta_{j,\ell}\left(w\right)=m_{k},\qquad k=0,1,\dots,2r-1,
\]
remains non-degenerate in this neighborhood. We conclude that the
coefficients $\left\{ \beta_{j,\ell}\left(w\left(\mu^{*}\right)\right)\right\} $
are multivalued algebraic functions, continuous in a neighborhood
of $\mu_{0}^{*}$.
\end{proof}

\section{\label{sec:real-prony}Real Prony space and hyperbolic polynomials}

In this section we shall restrict ourselves to the real case. Notice
that in many applications only real Prony systems are used. On the
other hand, considering the Prony problem over the real numbers significantly
simplifies some constructions. In particular, we can easily avoid
topological problems, related with the choice of the ordering of the
points $x_{1},\dots,x_{d}\in\CC.$ So in a definition of the real
Prony space $R{\cal P}_{d}$ we assume that the coordinates $x_{1},\ldots,x_{d}$
are taken with their natural ordering $x_{1}\leq x_{2}\leq\dots\leq x_{d}$.
Accordingly, the real Prony space $R{\cal P}_{d}$ is defined as the
bundle $(w,g),\ w\in\prod_{d}\subset{\mathbb{R}}^{d},g\in RV_{w}.$
Here $\prod_{d}$ is the prism in ${\mathbb{R}}^{d}$ defined by the
inequalities $x_{1}\leq x_{2}\leq\dots\leq x_{d}$, and $RV_{w}$
is the space of linear combinations with real coefficients of $\delta$-functions
and their derivatives with the support $\{x_{1},\ldots,x_{d}\},$
as in \prettyref{def:bundles}. The Prony, Stieltjes and Taylor maps
are the restrictions to the real case of the complex maps defined
above.

In this paper we just point out a remarkable connection of the real
Prony space and mapping with hyperbolic polynomials, and Vieta and
Vandermonde mappings studied in Singularity Theory (see \cite{arnold1986hpa,kostov1989geometric,kostov2006root,kostov2007root}
and references therein).

Hyperbolic polynomials (in one variable) are real polynomials $Q(z)=z^{d}+\sum_{j=1}^{d}\lambda_{j}z^{d-j},$
with all $d$ of their roots real. We denote by $\Gamma_{d}$ the
space of the coefficients $\Lambda=(\lambda_{1},\ldots,\lambda_{d})\subset{\mathbb{R}}^{d}$
of all the hyperbolic polynomials, and by $\hat{\Gamma}_{d}$ the
set of $\Lambda\in\Gamma_{d}$ with $\lambda_{1}=0,\ |\lambda_{2}|\leq1.$
Recalling \eqref{eq:sm-rank-equals-degree-denominator}, it is evident
that all hyperbolic polynomials appear as the denominators of the
irreducible fractions in the image of $R{\cal P}_{d}$ by $\SM$.
This shows, in particular, that the geometry of the boundary $\partial\Gamma$
of the hyperbolicity domain $\Gamma$ is important in the study of
the real Prony map $\PM$: it is mapped by $\PM$ to the boundary
of the solvability domain of the real Prony problem. This geometry
has been studied in a number of publications, from the middle of 1980s.
In \cite{kostov1989geometric} V. P. Kostov has shown that $\hat{\Gamma}$
possesses the Whitney property: there is a constant $C$ such that
any two points $\lambda_{1},\lambda_{2}\in\hat{\Gamma}$ can be connected
by a curve inside $\hat{\Gamma}$ of the length at most $C\Vert\lambda_{2}-\lambda_{1}\Vert$.
``Vieta mapping'' which associates to the nodes $x_{1}\leq x_{2}\leq\dots\leq x_{d}$
the coefficients of $Q(z)$ having these nodes as the roots, is also
studied in \cite{kostov1989geometric}. In our notations, Vieta mapping
is the composition of the Stieltjes mapping $\SM$ with the projection
to the coefficients of the denominator.

In \cite{arnold1986hpa} V.I.Arnold introduced and studied the notion
of maximal hyperbolic polynomial, relevant in description of $\hat{\Gamma}$.
Furthermore, the Vandermonde mapping ${\cal V}:{\mathbb{R}}^{d}\rightarrow{\mathbb{R}}^{d}$
was defined there by
\begin{eqnarray*}
\begin{cases}
y_{1}=a_{1}x_{1}+\ldots+a_{d}x_{d},\\
\dots\\
y_{d}=a_{1}x_{1}^{d}+\ldots+a_{d}x_{d}^{d},
\end{cases}
\end{eqnarray*}
with $a_{1},\ldots,a_{d}$ fixed. In our notations ${\cal V}$ is
the restriction of the Prony mapping to the pairs $(w,g)\in R{\cal P}_{d}$
with the coefficients of $g$ in the standard basis of $RV_{w}$ fixed.
It was shown in \cite{arnold1986hpa} that for $a_{1},\ldots,a_{d}>0$
${\cal V}$ is a one-to-one mapping of $\prod_{d}$ to its image.
In other words, the first $d$ moments uniquely define the nodes $x_{1}\leq x_{2}\leq\dots\leq x_{d}$.
For $a_{1},\ldots,a_{d}$ with varying signs, this is no longer true
in general. This result is applied in \cite{arnold1986hpa} to the
study of the colliding configurations.

Next, the ``Vandermonde varieties'' are studied in \cite{arnold1986hpa},
which are defined by the equations
\begin{eqnarray*}
\begin{cases}
a_{1}x_{1}+\ldots+a_{d}x_{d} & =\alpha_{1},\\
 & \dots\\
a_{1}x_{1}^{\ell}+\ldots+a_{d}x_{d}^{\ell} & =\alpha_{\ell}.
\end{cases} &  & \ell\leqslant d.
\end{eqnarray*}
It is shown that for $a_{1},\ldots,a_{d}>0$ the intersections of
such varieties with $\prod_{d}$ are either contractible or empty.
Finally, the critical points of the next Vandermonde equation on the
Vandermond variety are studied in detail, and on this base a new proof
of Kostov's theorem is given.

We believe that the results of \cite{arnold1986hpa,kostov1989geometric}
and their continuation in \cite{kostov2006root,kostov2007root} and
other publications are important for the study of the Prony problem
over the reals, and we plan to present some results in this direction
separately.

\appendix

\section{\label{app:proof-pade-solvability}Proof of \prettyref{thm:pade-solvability}}

Recall that we are interested in finding conditions for which the
Taylor mapping $\TM:\;{\cal S}_{d}\to{\cal T}_{d}$ is invertible.
In other words, given
\[
S\left(z\right)=\sum_{k=0}^{2d-1}m_{k}\left(\frac{1}{z}\right)^{k+1},
\]
we are looking for a rational function $R\left(z\right)\in{\cal S}_{d}$
such that
\begin{equation}
S\left(z\right)-R\left(z\right)=\frac{d_{1}}{z^{2d+1}}+\frac{d_{2}}{z^{2d+2}}+\dots.\label{eq:Pade1}
\end{equation}

Write $R\left(z\right)=\frac{P\left(z\right)}{Q\left(z\right)}$ with
$Q\left(z\right)=\sum_{j=0}^{d}c_{j}z^{j}$ and $P\left(z\right)=\sum_{i=0}^{d-1}b_{i}z^{i}$.
Multiplying \eqref{eq:Pade1} by $Q\left(z\right)$, we obtain
\begin{equation}
Q\left(z\right)S\left(z\right)-P\left(z\right)=\frac{e_{1}}{z^{d+1}}+\frac{e_{2}}{z^{d+2}}+\dots.\label{eq:Pade2}
\end{equation}

\begin{prop}
The identity \eqref{eq:Pade2}, considered as an equation on $P$
and $Q$ with $\deg P<\deg Q\leq d$, always has a solution.\end{prop}
\begin{proof}
Substituting the expressions for $S,\ P$ and $Q$ into \eqref{eq:Pade2}
we get
\begin{equation}
\left(c_{0}+c_{1}z+\dots+c_{d}z^{d}\right)\left(\frac{m_{0}}{z}+\frac{m_{1}}{z^{2}}+\dots\right)-b_{0}-\dots-b_{{d-1}}z^{{d-1}}=\frac{e_{1}}{z^{{d+1}}}+\dots.\label{eq:comp.coef}
\end{equation}
The highest degree of $z$ in the left hand side of \eqref{eq:comp.coef}
is $d-1$. So equating to zero the coefficients of $z^{s}$ in \eqref{eq:comp.coef}
for $s=d-1,\dots,-d$ we get the following systems of equations:
\begin{equation}
\begin{bmatrix}0 & 0 & 0 & m_{0}\\
0 & 0 & m_{0} & m_{1}\\
\adots & \adots\\
m_{0} & m_{1} & \dots & m_{d-1}
\end{bmatrix}\begin{bmatrix}c_{1}\\
c_{2}\\
\vdots\\
c_{d}
\end{bmatrix}=\begin{bmatrix}b_{d-1}\\
b_{d-2}\\
\vdots\\
b_{0}
\end{bmatrix}.\tag{\ref{app:proof-pade-solvability}.\ensuremath{\star}}\label{eq:star}
\end{equation}
From this point on, the equations become homogeneous:
\begin{equation}
\begin{bmatrix}m_{0} & m_{1} & \dots & m_{d}\\
m_{1} & m_{2} & \dots & m_{d+1}\\
\adots & \adots\\
m_{d-1} & m_{d} & \dots & m_{2d-1}
\end{bmatrix}\begin{bmatrix}c_{0}\\
c_{1}\\
\vdots\\
c_{d}
\end{bmatrix}=\begin{bmatrix}0\\
0\\
\vdots\\
0
\end{bmatrix}.\tag{\ref{app:proof-pade-solvability}.\ensuremath{\star\star}}\label{eq:double-star}
\end{equation}
The homogeneous system \eqref{eq:double-star} has the Hankel-type
$d\times\left(d+1\right)$ matrix $\M_{d}=\left(m_{i+j}\right)$ with
$0\leqslant i\leqslant d-1$ and $0\leqslant j\leqslant d$. This
system has $d$ equations and $d+1$ unknowns $c_{0},\dots,c_{d}$.
Consequently, it always has a nonzero solution $c_{0},\dots,c_{d}$.
Now substituting these coefficients $c_{0},\dots,c_{d}$ of $Q$ into
the equations \eqref{eq:star} we find the coefficients $b_{0},\dots,b_{d-1}$
of the polynomial $P$, satisfying \eqref{eq:star}. Notice that if
$c_{j}=0$ for $j\geqslant\ell+1$ then it follows from the structure
of the equations \eqref{eq:star} that $b_{j}=0$ for $j\geq\ell$.
Hence these $P,Q$ provide a solution of \eqref{eq:Pade2}, satisfying
$\deg P<\deg Q\leq d,$ and hence belonging to ${\cal S}_{d}.$
\end{proof}
However, in general \eqref{eq:Pade2} does not imply \eqref{eq:Pade1}.
This implication holds only if $\deg Q=d$. The following proposition
describes a possible ``lost of accuracy'' as we return from \eqref{eq:Pade2}
to \eqref{eq:Pade1} and $\deg Q<d$:
\begin{prop}
\label{prop:rank.pade}Let \eqref{eq:Pade2} be satisfied with the
highest nonzero coefficient of $Q$ being $c_{\ell},\ \ell\le d$.
Then
\begin{equation}
S(z)-\frac{{P(z)}}{{Q(z)}}=\frac{d_{1}}{z^{{d+\ell+1}}}+\frac{d_{2}}{z^{{d+\ell+2}}}+\dots.\label{eq:Pade3}
\end{equation}
\end{prop}
\begin{proof}
We notice that if the leading nonzero coefficient of $Q$ is $c_{\ell}$
then we have
\[
\frac{1}{Q}=\frac{1}{{z^{\ell}}}(\frac{1}{{c_{\ell}+\frac{c_{\ell-1}}{z}+\dots}})=\frac{1}{{z^{\ell}}}(f_{0}+f_{1}\frac{1}{z}+\dots).
\]
So multiplying \eqref{eq:Pade2} by $\frac{1}{Q}$ we get \eqref{eq:Pade3}.
\end{proof}
\begin{minipage}[t]{1\columnwidth}%
\end{minipage}
\begin{proof}[Proof of \prettyref{thm:pade-solvability}]
Assume that the rank of $\tilde{M}_{d}$ is $r\leq d,$ and that
$|M_{r}|\ne0.$ Let us find a polynomial $Q(z)$ of degree $r$ of
the form $Q(z)=z^{r}+\sum_{j=0}^{r-1}c_{j}z^{j},$ whose coefficients
satisfy system \eqref{eq:double-star}. Put $\vec c_{r}=(c_{0},\dots,c_{r-1},1)^{T}$
and consider a linear system $\tilde{M}_{r}\vec c_{r}=0$. Since by
assumptions $|M_{r}|\ne0,$ this system has a unique solution. Extend
this solution by zeroes, i.e. put $\vec c_{d}=(c_{0},\dots,c_{r-1},1,0,\dots,0)^{T}.$
We want $\vec c_{d}$ to satisfy \eqref{eq:double-star}, which is
$\tilde{M}_{d}\vec c_{d}=0$. This fact is immediate for the first
$r$ rows of $\tilde{M}_{d}$. But since the rank of $\tilde{M}_{d}$
is $r$ by the assumption, its other rows are linear combinations
of the first $r$ ones. Hence $\vec c_{d}$ satisfies \eqref{eq:double-star}.

Now the equations \eqref{eq:star} produce a polynomial $P(z)$ of
degree at most $r-1$. So we get a rational function $R(z)=\frac{{P(z)}}{{Q(z)}}\in{\cal S}_{r}\subseteq{\cal S}_{d}$
which solves the Padé problem \eqref{eq:Pade2}, with $\deg Q(z)=r$.
Write $R(z)=\sum_{k=0}^{\infty}\alpha_{k}(\frac{1}{z})^{k+1}$. By
\prettyref{prop:rank.pade} we have $m_{k}=\alpha_{k}$ till $k=d+r-1$.

Now, the Taylor coefficients $\alpha_{k}$ of $R(z)$ satisfy a linear
recurrence relation
\begin{equation}
m_{k}=-\sum_{{s=1}}^{r}c_{s}m_{{k-s}},\qquad k=r,r+1,\dots.\label{eq:lin.rec}
\end{equation}
Considering the rows of the system $\tilde{M}_{d}\vec c_{d}=0$ we
see that $m_{k}$ satisfy the same recurrence relation \eqref{eq:lin.rec}
till $k=d+r-1$ (we already know that $m_{k}=\alpha_{k}$ till $k=d+r-1$).
We shall show that in fact $m_{k}$ satisfy \eqref{eq:lin.rec} till
$k=2d-1.$

Consider a $d\times r$ matrix $\bar{M}_{d}$ formed by the first
$r$ columns of $M_{d}$, and denote its row vectors by $\vec v_{i}=(m_{i,0},\dots,m_{i,r-1}),\ i=1,\dots,d-1$.
The vectors $\vec v_{i}$ satisfy
\begin{equation}
\vec v_{i}=-\sum_{{s=1}}^{r}c_{s}\vec v_{{i-s}},\quad i=r,\dots,d-1,\label{eq:row.lr}
\end{equation}
since their coordinates satisfy \eqref{eq:lin.rec} till $k=d+r-1$.
Now $\vec v_{0},\dots,\vec v_{r-1}$ are linearly independent, and
hence each $\vec v_{i},\ i=r,\dots,d-1,$ can be expressed as
\begin{equation}
\vec v_{i}=\sum_{{s=0}}^{{r-1}}\gamma_{{i,s}}\vec v_{s}.\label{eq:rows.md}
\end{equation}
Denote by $\vec{\tilde{v}}_{i}=(m_{i,0},\dots,m_{i,d}),\ i=1,\dots,d-1$
the row vectors of $\tilde{M}_{d}$. Since by assumptions the rank
of $\tilde{M}_{d}$ is $r$, the vectors $\vec{\tilde{v}}_{i}$ can
be expressed through the first $r$ of them exactly in the same form
as $\vec v_{i}$:
\begin{equation}
\vec{\tilde{v}}_{i}=\sum_{{s=0}}^{{r-1}}\gamma_{{i,s}}\vec{\tilde{v}}_{s},\quad i=r,\dots,d-1.\label{eq:rows.md1}
\end{equation}
Now the property of a system of vectors to satisfy the linear recurrence
relation \eqref{eq:row.lr} depends only on the coefficients $\gamma_{i,s}$
in their representation \eqref{eq:rows.md} or \eqref{eq:rows.md1}.
Hence from \eqref{eq:row.lr} we conclude that the full rows $\vec{\tilde{v}}_{i}$
of $\tilde{M}_{d}$ satisfy the same recurrence relation. Coordinate-wise
this implies that $m_{k}$ satisfy \eqref{eq:lin.rec} till $k=2d-1,$
and hence $m_{k}=\alpha_{k}$ till $k=2d-1.$ So $R(z)$ solves the
original \prettyref{prob:pade}.

In the opposite direction, assume that $R(z)$ solves \prettyref{prob:pade},
and that the representation $R(z)=\frac{{P(z)}}{{Q(z)}}\in{\cal S}_{r}\subset{\cal S}_{d}$
is irreducible, i.e. $\deg Q=r$. Write $Q(z)=z^{r}+\sum_{j=0}^{r-1}c_{j}z^{j}$.
Then $m_{k}$, being the Taylor coefficients of $R(z)$ till $k=2d-1$,
satisfy a linear recurrence relation \eqref{eq:lin.rec}: $m_{k}=-\sum_{s=1}^{r}c_{s}m_{k-s},\ k=r,r+1,\dots,2d-1.$
Applying this relation coordinate-wise to the rows of $\tilde{M}_{d}$
we conclude that all the rows can be linearly expressed through the
first $r$ ones. So the rank of $\tilde{M}_{d}$ is at most $r$.

It remains to show that the left upper minor $|M_{r}|$ is non-zero,
and hence the rank of $\tilde{M}_{d}$ is exactly $r$.

By \prettyref{prop:pade-correspondence}, if the decomposition of
$R\left(z\right)$ in the standard basis is
\[
R\left(z\right)=\sum_{j=1}^{s}\sum_{\ell=1}^{d_{j}}a_{j,\ell-1}\frac{\left(-1\right)^{\ell-1}\left(\ell-1\right)!}{\left(z-x_{j}\right)^{\ell}},
\]
where $\sum_{j=1}^{s}d_{j}=r$ and $\left\{ x_{j}\right\} $ are pairwise
distinct, then the Taylor coefficients of $R\left(z\right)$ are given
by \eqref{eq:confluent equation_prony_system}. Clearly, we must have
$a_{j,d_{j}-1}\neq0$ for all $j=1,\dots,s$, otherwise $\deg Q<r$,
a contradiction. Now consider the following well-known representation
of $M_{r}$ as a product of three matrices (see e.g. \cite{batenkov2011accuracy}):
\begin{equation}
M_{r}=V\left(x_{1},d_{1},\dots,x_{s},d_{s}\right)\times\diag\left\{ A_{j}\right\} _{j=1}^{s}\times V\left(x_{1},d_{1},\dots,x_{s},d_{s}\right)^{T},\label{eq:mr-decomposition-product}
\end{equation}
where $V\left(\dots\right)$ is the confluent Vandermonde matrix \eqref{eq:confluent-vandermonde-def}
and each $A_{j}$ is the following $d_{j}\times d_{j}$ block:
\[
A_{j}\isdef\begin{bmatrix}a_{j,0} & a_{j,1} & \cdots & \cdots & a_{j,d_{j}-1}\\
a_{j,1} &  &  & {d_{j}-1 \choose d_{j}-2}a_{j,d_{j}-1} & 0\\
\cdots &  &  & \cdots & 0\\
 & {d_{j}-1 \choose 2}a_{j,d_{j}-1} & 0 & \cdots & 0\\
a_{j,d_{j}-1} & 0 & \cdots & \cdots & 0
\end{bmatrix}.
\]
The formula \eqref{eq:mr-decomposition-product} can be checked by
direct computation. Since $\left\{ x_{j}\right\} $ are pairwise distinct
and $a_{j,d_{j}-1}\neq0$ for all $j=1,\dots,s$, we immediately conclude
that $\left|M_{r}\right|\neq0$.

This finishes the proof of \prettyref{thm:pade-solvability}.
\end{proof}
\bibliographystyle{plain}
\bibliography{../../../bibliography/all-bib}

\end{document}